\documentclass{birkjour}
\usepackage{amsmath}
\usepackage{mathrsfs}
\usepackage{bm}
\usepackage{amsthm}
\usepackage{amssymb}
\usepackage{enumerate}

\numberwithin{equation}{section}

\begingroup
\newtheorem{theorem}{Theorem}[section]
\newtheorem{lemma}[theorem]{Lemma}
\newtheorem{proposition}[theorem]{Proposition}

\endgroup

\theoremstyle{definition}

\newtheorem{remark}[theorem]{Remark}

\allowdisplaybreaks[3]

\swapnumbers
\newtoks \prt
\numberwithin{equation}{section}
\newtheorem{proclaim}{\the \prt}[section]
\def\bh{\begin{proclaim}}
\def\eh{\end{proclaim}}

\newtoks \defk
\theoremstyle{definition}
\numberwithin{equation}{section}
\newtheorem{definice}{ \the \defk}[section]
\def\bde{\begin{definice}}
\def\ede{\end{definice}}
\newtoks \pozn
\theoremstyle{remark}

\def\sup{\operatorname{sup}}

\def\min{\operatorname{min}}

\def\max{\operatorname{max}}

\def\O{\mathcal O}

\def\q{\mathbf q}
\def\R{\mathbb R}

\def\S{\mathbb S}
\def\\tau {\mathcal \tau }

\def\u{\mathbf u}
\def\U{\mathbf U}

\def \and {\ \&\ }

\newcommand{\dx}{\,\mathrm{d}x}

\def\u{\mathbf{u}}
\def\ep{\varepsilon}

\def\n{\mathbf{n}}
\def\ve{\mathbf{v}}

\def\de{\partial}
\newcommand{\abs}[1]{\ensuremath{\left| #1 \right|}}

\renewcommand{\div}{{\rm div}}

\newcommand{\bFormula}[1]{
\begin{equation} \label{#1}}
\newcommand{\eF}{\end{equation}}

\usepackage[usenames,dvipsnames]{color}
\newcommand{\EE}{\mathcal{E}}

\def\eqn#1$$#2$${\begin{equation}\label#1#2\end{equation}}

\begin{document}

\title[Dimension Reduction for the Full System]{Dimension reduction for the full Navier-Stokes-Fourier system}

\author{Jan B\v rezina}
\address{Tokyo Institute of Technology,\br 2-12-1 Ookayama, Megoru-ku,\br Tokyo, 152-8550\br Japan}
\email{brezina@math.titech.ac.jp}
\author{Ond\v rej Kreml}
\address{Institute of Mathematics of the Czech Academy of Sciences,\br \v Zitn\' a 25,\br 115 67 Praha 1,\br Czech Republic}
\email{kreml@math.cas.cz}
\author{V\'aclav M\'acha}\address{Industry-University Research Center,\br Yonsei University,\br 50 Yonsei-ro Seodaemun-gu,\br Seoul, 03722,\br Republic of Korea}
\email{macha@math.cas.cz}
\thanks{O.K. acknowledges the support of the GA\v CR (Czech Science Foundation) project GA13-00522S in the general framework of RVO: 67985840. The research of V.M. has been supported by the grant NRF-20151009350.}
\date{}

\begin{abstract}
It is well known that the full Navier-Stokes-Fourier system does not possess a strong solution in three dimensions which causes problems in applications. However, when modeling the flow of a fluid in a thin long pipe, the influence of the cross section can be  neglected and the  flow is basically one-dimensional. This allows us to deal with strong solutions which are more  convenient for numerical computations. The goal of this paper is to provide a rigorous justification of this approach. Namely, we prove that any suitable weak solution to the three-dimensional NSF system tends to a strong solution to the one-dimensional system as the thickness of the pipe tends to zero.
\end{abstract}

\maketitle

\section{Introduction}
As introduced in \cite[Chapter 1]{Fei2}, governing equations for a flow of a general compressible viscous heat conducting fluid in a domain of $\mathbb R^3$ read as

\begin{eqnarray}
\label{continuityeq-main} \partial_t \rho + \div_x (\rho \u) &=& 0,\\
\label{momentumeq-main} \partial_t (\rho \u) + \div_x (\rho \u \otimes \u) + \nabla_x p (\rho, \theta) -\div_x \mathbb{S}(\theta, \nabla_x \u) &=& 0,\\
\label{entropyeq-main} \partial_t(\rho s(\rho, \theta)) + \div_x (\rho s(\rho, \theta)\u) + \div_x \left(\frac{\q(\theta, \nabla_x \theta)}{\theta}\right) &=& \sigma,
\end{eqnarray}
where $(\rho, \u, \theta)$ stand for the unknown fluid mass density, the  velocity field and the temperature respectively, $p$ is the pressure, $s$ is the entropy, $\q$ is the heat flux, $\sigma$ is the entropy production rate and $\mathbb S$ represents the stress tensor.

We consider a family of shrinking domains $\Omega_\ep$ of the form 

\begin{equation*} \Omega_\ep = Q_\varepsilon\times (0,1),  \ \ \ Q_\varepsilon = \varepsilon Q, 
\end{equation*}
where $Q$ is an open rectangular domain in $ \mathbb R^2$ and $\ep >0$. Under suitable conditions on the initial data it is natural to expect that weak solutions $(\rho_\ep, \u_\ep, \theta_\ep)$ of \eqref{continuityeq-main}--\eqref{entropyeq-main}  on $\Omega_\ep$ tend, as $\ep \rightarrow 0$, to a 
 classical solution $(\tilde{\rho},\tilde{u},\tilde{\theta})$ of the one-dimensional system

 \begin{align}
\label{continuityeq-1D}
\partial_t \tilde{\rho} + \partial_y (\tilde{\rho} \tilde{u}) &= 0,\\
\label{momentumeq-1D} 
 \partial_t (\tilde{\rho} \tilde{u}) + \partial_y (\tilde{\rho} \tilde{u}^2) + \partial_y p(\tilde{\rho}, \tilde{\theta})-\de_y[ \tilde{S}(\tilde{\theta}, \de_y \tilde{u})] &= 0,\\ %\ \ \ \nu_i = \frac43 \mu_i + \eta_i, \ i=0,1,
\label{entropyeq-1D}
\partial_t(\tilde{\rho} s(\tilde{\rho}, \tilde{\theta})) + \de_y (\tilde{\rho} s(\tilde{\rho}, \tilde{\theta})\tilde{u}) + \de_y \left(\frac{q(\tilde{\theta}, \de_y \tilde{\theta})}{\tilde{\theta}}\right)\nonumber &=\\ \frac1{\tilde{\theta}}\bigg(\tilde{S}(\tilde{\theta},\de_y \tilde{u})&\de_y \tilde{u} - \frac{q(\tilde{\theta}, \de_y \tilde{\theta}) \de_y \tilde{\theta}}{\tilde{\theta}}\bigg),
\end{align}
where %$\nu_0$ and $\nu_1$ stand for the viscosity coefficients  which we specify later
$\tilde{S}(\tilde{\theta},\de_y \tilde{u})$ is naturally related to the three-dimensional stress tensor $\mathbb{S}$ and similarly $q$ to the heat flux vector $\q$, see \eqref{eq:1Dstress}--\eqref{eq:1Dheat}. Hereinafter we use the notation $x = (x_h,y)\in \mathbb R^3,\ x_h=(x_1,x_2)\in \mathbb R^2, \ y=x_3\in \mathbb R$ and denote the derivative in $x_3$  by $\partial_y$. In this paper we give a rigorous justification of the convergence $(\rho_\ep, \u_\ep, \theta_\ep)\rightarrow (\tilde{\rho}, (0,0,\tilde{u}), \tilde{\theta})$.

As far as we know, the limit passage for  heat conductive fluids has not yet been rigorously investigated and there is only a handful of results on related problems. Since incompressibility in one dimension does not allow any movement, such limit makes a little sense for 1D incompressible flows. However, dimension reduction to 2D-planar flows was examined in \cite{IfRaSe}, \cite{RaSe}, \cite{RaSe2}, \cite{RaSe3} -- see also references given therein.
 The case of a compressible barotropic fluid was studied by Vod\'ak \cite{Vod} and later by Bella et. al. \cite{BeFeNo}.

The paper is organized as follows. In Section 2 we introduce detailed description of our problem. In Section 3 we present the concepts of a weak and a classical solution for 3D and 1D system, respectively, and discuss their existence. The main result is stated in Section 4. Section 5 contains preliminary calculations which are later used in Sections 6 and 7 in order to establish the proof of the main theorem from Section 4.

\section{Setting of the problem}

\subsection{Structural hypothesis for the 3D problem}

For given $\varepsilon >0$ the system \eqref{continuityeq-main}--\eqref{entropyeq-main} on $\Omega_\ep$ is complemented by the initial conditions

$$
\rho_\ep(0,\cdot) = \rho_{0,\ep},\ \rho_\ep\u_\ep(0,\cdot)=(\rho\u)_{0,\ep}\ \mbox{and} \  \theta_\ep(0,\cdot) = \theta_{0,\ep},
$$
such that the integral averages  over $Q_\ep$ converge weakly (with respect to $y$) in $L^1(0,1)$
to some limit as $\ep \to 0$. For the precise conditions see the statement of Theorem \ref{mainthm}.

We suppose that the viscous stress tensor $\mathbb S$ is a linear function of the velocity gradient and therefore described by the Newton's law

\eqn{const-stress} $$\mathbb{S}(\theta, \nabla_x \u) = \mu(\theta) \left(\nabla_x \u+ \nabla_x^T  \u - \frac23 \div_x \u \mathbb{I}\right) + \eta (\theta) \div_x \u \mathbb{I},$$
with the shear viscosity coefficient $\mu(\theta) > 0$ and the bulk viscosity coefficient $\eta (\theta) \geq 0$ satisfying %for simplicity of our presentation relations

%\todo{how about some $0$ coefficients?}

\eqn{stress-cond} $$\mu(\theta) = \mu_0 + \mu_1 \theta, \ \ \ \eta(\theta) = \eta_0 + \eta_1 \theta, \ \ \ \mu_0,\mu_1 > 0,\quad \eta_0,\eta_1 \geq 0.$$ 
The heat flux $\q$ satisfies the Fourier's law
\eqn{const-heatflux} $$\q(\theta, \nabla_x \theta ) = -\kappa(\theta) \nabla_x \theta, $$
where we %again for simplicity of our presentation
assume the following form of $\kappa(\theta)$

%\todo{how about diminishing some $\kappa_i$ or viscosity in the limit?}

\eqn{heatflux-cond} $$\kappa(\theta) = \kappa_0  + \kappa_2 \theta^2 +  \kappa_3  \theta^3, \ \  \kappa_i > 0,\ i = 0,2, 3.$$

The system of equations \eqref{continuityeq-main}--\eqref{entropyeq-main} with the constitutive relations \eqref{const-stress} and  \eqref{const-heatflux}
 is called the Navier-Stokes-Fourier system.

Equations \eqref{continuityeq-main}--\eqref{entropyeq-main} are supplemented with the
conservative boundary condition 

\begin{equation} \label{eq:BCq} \q\cdot \n|_{\partial \Omega_\varepsilon} = 0,\end{equation}
and the complete slip boundary conditions

%\todo{to use the new Poincare-Korn we need extra Dirichlet on top and bottom}

\begin{equation}\label{eq:BC}
\u\cdot \n|_{\partial \Omega_\varepsilon} = 0, \ \ \ [\mathbb{S}(\nabla_x \u)\cdot \n]\times \n|_{\partial \Omega_\varepsilon} = 0,\end{equation}
where the symbol $\n$ denotes the outer normal vector. It is worth pointing out that the complete slip boundary conditions are suitable for a dimension-reduction as the no-slip boundary conditions yield only a trivial solution in the asymptotic limit $\ep \to 0$.

The concept of a weak solution to the Navier-Stokes-Fourier system  based on the Second law of thermodynamics was introduced in \cite{Ducomet}. The weak solutions satisfy the field equations \eqref{continuityeq-main}--\eqref{entropyeq-main} in the sense of distributions where the entropy production rate $\sigma$ is a non-negative measure,

\eqn{entropyproduction} $$\sigma \geq \frac1{\theta}\left(\mathbb{S} (\theta, \nabla_x \u): \nabla_x \u - \frac{\q(\theta, \nabla_x \theta) \nabla_x \theta}{\theta}\right).$$
In order to compensate for the lack of information resulting from the inequality sign in \eqref{entropyproduction} the system is supplemented with the total energy balance,

 $$\frac{\partial}{\partial t} \int_{\Omega_\varepsilon} \left(\frac12 \rho |\u|^2 + \rho e (\rho, \theta)\right)dx = 0,$$
where $e = e (\rho, \theta)$ is the (specific) internal energy. Under these circumstances it can be shown (see  \cite[Chapter 3]{Fei2}) that any sufficiently smooth weak solution of \eqref{continuityeq-main}--\eqref{entropyeq-main} satisfies the standard relation

$$\sigma = \frac1{\theta}\left(\mathbb{S} (\theta, \nabla_x \u): \nabla_x \u - \frac{\q(\theta, \nabla_x \theta) \nabla_x \theta}{\theta}\right).$$

The proof of our main theorem is based on the method of the relative entropy (see \cite{Carrillo}, \cite{Dafermos}, \cite{Saint-Raymond}), represented  by the quantity

\begin{equation} \label{entfunc} \EE(\rho, \theta|r, \Theta)  = H^\Theta(\rho, \theta) - \partial_\rho H^\Theta(r, \Theta)(\rho - r) - H^\Theta(r, \Theta),
\end{equation}
where $H^\Theta(\rho, \theta)$ is the  thermodynamic potential called the ballistic free energy

\begin{equation*}H^\Theta(\rho, \theta) = \rho e (\rho, \theta) - \Theta \rho s(\rho, \theta),\end{equation*}
introduced by Gibbs  and discussed more recently by Ericksen \cite{Ericksen}. 

We assume that the thermodynamic functions $p$, $e$ and $s$ are interrelated through the Gibbs' equation 

\eqn{gibbs} $$ \theta D s (\rho, \theta) = De(\rho, \theta) + p(\rho,\theta) D\left(\frac1\rho\right).$$

The subsequent analysis leans essentially on the thermodynamic stability of the fluid  expressed through

\eqn{thermostabhypo} $$ \frac{\partial p (\rho, \theta)}{\partial \rho } > 0, \ \ \ \frac{\partial e (\rho, \theta)}{\partial \theta } > 0 \mbox{ for all } \rho, \theta >0.$$

%Our method is based on a kind of {\it total dissipation balance} stated in terms of the quantity
%
%$$  I = \int_{\Omega_\ep} \left(\frac12 \rho |\u - \U|^2 + \EE(\rho, \theta| r, \Theta)\right)dx,$$
%where $\{\rho, \theta, \u\}$ is a weak solution to the Navier-Stokes-Fourier system and  $\{ \U, r, \Theta\}$ is any triple of functions with an appropriate regularity and boundary conditions on $\U$. 

Motivated by the existence theory developed in \cite[Chapter 3]{Fei2} we assume that the pressure $p = p (\rho, \theta)$ can be written in the form 

\eqn{const-pressure} $$p(\rho, \theta) =  \theta^\frac52  P\left(\frac{ \rho}{\theta^\frac32}\right) +  \frac{a}3 \theta^4, \ \ \ a >0,$$
where

\begin{equation} P \in C^1[0,\infty),\  P(0) = 0,\  P'(Z) > 0 \mbox{ for all } Z \geq 0. \end{equation}

In agreement with Gibbs' relation \eqref{gibbs}, the specific internal energy can be taken as

\begin{equation} e(\rho, \theta) =  \frac32  \frac{\theta^\frac52}{\rho}  P\left(\frac{ \rho}{\theta^\frac32}\right) +  a \frac{\theta^4}{\rho}.\end{equation}

Furthermore, by virtue of the second inequality in the thermodynamic stability hypotheses \eqref{thermostabhypo}, we have 

\eqn{P-cond-2} $$0 < \frac{\frac53 P(Z) - ZP'( Z)}{Z} < c, \ \ \ \mbox{ for all } Z >0. $$ 
In particular, \eqref{P-cond-2} implies that the function $ Z\mapsto \frac{P(Z)}{Z^\frac53}$ is  decreasing and we suppose that

\begin{equation}  \lim_{Z \to \infty} \frac{P(Z)}{Z^\frac53} = P_\infty > 0.\end{equation}

Finally, the formula for the (specific) entropy reads from the Gibbs' equation \eqref{gibbs} as

\begin{equation} s(\rho,\theta) = S\left(\frac{\rho}{\theta^\frac32}\right) + \frac{4a}3 \frac{\theta^3}\rho,\end{equation}
where, in accordance with the Third law of thermodynamics,

 \eqn{S-cond} $$ 
S'(Z) = -\frac32 \frac{\frac53  P( Z) - ZP'( Z)}{Z^2} < 0.$$

\subsection{Structural hypothesis for the 1D problem}

Since we are interested in smooth solutions of the 1D equations, we complement the system of equations \eqref{continuityeq-1D}-\eqref{entropyeq-1D} with the initial conditions

\begin{equation}\label{1DIC}
\tilde{\rho}(0,\cdot) = \tilde{\rho}_0, \ \tilde{u}(0,\cdot) = \tilde{u}_0, \ \tilde{\theta}(0,\cdot) = \tilde{\theta}_0,
\end{equation}
with $\tilde{\rho}_0 \geq c > 0$,  $\tilde{\theta}_0 \geq c > 0$ and $\tilde{u}_0$ being smooth functions.

The form of the stress tensor $\mathbb{S}$ and the heat flux $\q$ naturally yields their one-dimensional counterparts, namely we expect that in the limit we recover
\begin{equation}\label{eq:1Dstress}
\tilde{S}(\tilde{\theta},\de_y \tilde{u}) = (\nu_0 + \nu_1\tilde{\theta})\de_y \tilde{u},
\end{equation}
with the viscosity coefficients
\begin{equation}\label{eq:1Dvisc}
\nu_i = \frac 43 \mu_i + \eta_i, \qquad i = 0,1,
\end{equation}
and
\begin{equation}\label{eq:1Dheat}
q(\tilde{\theta},\de_y \tilde{\theta}) = -\kappa(\tilde{\theta})\de_y\tilde{\theta}.
\end{equation}

\section{Concepts of solutions} 
\subsection{Weak solutions to the $3$D-system} \label{weaksol-section}

Let $T >0$ and $\Omega \subset \R^3$ be  a bounded Lipschitz domain. We say that a triple $(\rho,  \u,\theta )$ is a weak solution to the Navier-Stokes-Fourier system \eqref{continuityeq-main}--\eqref{entropyeq-main} with \eqref{const-stress}--\eqref{entropyproduction} in $(0,T)\times \Omega$ emanating from the initial data

$$ \rho(0,\cdot) = \rho_0, \ \ \ \rho \u(0,\cdot) = (\rho\u)_0,\ \ \
 \rho s(\rho, \theta) (0,\cdot) = \rho_0 s(\rho_0, \theta_0), \ \rho_0 \geq 0, \ \theta_0 > 0, $$
if: 
\begin{itemize}
\item the density and the absolute temperature satisfy $\rho(t,x) \geq 0$, $ \theta(t,x) > 0$ for almost all $(t,x) \in (0,T )\times \Omega$, $\rho \in C_{weak}([0,T ];L^\frac53(\Omega))$, $\rho \u \in C_{weak}([0,T ]; L^\frac54 (\Omega;\R^3))$, $\theta \in L^\infty(0,T ; L^4(\Omega))\cap L^2(0,T ;W^{1,2}(\Omega))$ and $\u \in L^2(0,T ;W^{1,2} (\Omega;\R^3))$, $\u \cdot \n |_{\de \Omega} =0$;

\item  equation \eqref{continuityeq-main} is replaced by a family of integral identities

$$\int_{\Omega} \rho(\tau , \cdot)\varphi(\tau , \cdot)dx - \int_{\Omega} \rho_0 \varphi(0,\cdot)dx = \int_0^\tau  \int_{\Omega} (\rho \de_t \varphi + \rho \u \cdot \nabla_x \varphi )dxdt,$$
for any $\varphi \in C^1([0,T]\times \overline {\Omega})$ and any $\tau  \in [0,T ]$ ;

\item the momentum equation \eqref{momentumeq-main} is satisfied in the sense of distributions, specifically,

$$ \int_{\Omega} \rho \u (\tau , \cdot) \varphi(\tau , \cdot)dx - \int_{\Omega} (\rho\u)_0  \varphi(0, \cdot)dx $$

$$=\int_0^\tau  \int_{\Omega} (\rho \u  \de_t \varphi + \rho \u \otimes \u :\nabla_x \varphi + p(\rho, \theta) \div_x \varphi - \S (\theta, \nabla_x \u):\nabla_x \varphi)dxdt,$$
for any $\varphi \in C^1([0,T]\times \overline {\Omega};\R^3)$, $\varphi\cdot \n|_{\partial \Omega} = 0$ and any $\tau  \in [0,T ]$;

\item the entropy balance \eqref{entropyeq-main}, \eqref{entropyproduction} is replaced by a family of integral inequalities

 $$ \int_{\Omega} \rho_0 s(\rho_0, \theta_0) \varphi(0, \cdot) dx - \int_{\Omega} \rho s(\rho, \theta) ( \tau , \cdot) \varphi(\tau , \cdot) dx $$

$$+ \int_0^\tau  \int_{\Omega} \frac1\theta\left(\S(\theta, \nabla_x \u): \nabla_x \u - \frac{ \q (\theta, \nabla_x \theta) \cdot \nabla_x \theta}\theta \right)\varphi dx dt$$

$$\leq -\int_0^\tau  \int_{\Omega} \left(\rho s(\rho, \theta)\de_t \varphi + \rho s(\rho, \theta) \u \cdot \nabla_x \varphi + \frac{ \q (\theta, \nabla_x \theta) \cdot \nabla_x \varphi}\theta \right) dx dt,$$
for any  $\varphi \in C^1([0,T]\times \overline {\Omega})$, $\varphi \geq 0$ and almost all $\tau  \in [0,T ]$;

\item the total energy is conserved

$$ \int_{\Omega} \left(\frac12 \rho |\u|^2 + \rho e (\rho, \theta)\right)(\tau , \cdot) dx = \int_{\Omega} \left(\frac{1}{2\rho_0} |(\rho\u)_0|^2 + \rho_0 e (\rho_0, \theta_0)\right)dx ,$$
for almost all $\tau  \in [0,T ]$.

\end{itemize}  

The existence of global-in-time weak solutions to the three dimensional Navier-Stokes-Fourier system was established in \cite[Theorem 3.1]{Fei2}. It reads as follows.

\begin{theorem}\label{t:weaksolex}
Let $\Omega \subset \mathbb{R}^3$ be a bounded domain of class $C^{2,\nu}$, $\nu \in (0,1)$. Assume that $\rho_0 \in L^{5/3}(\Omega)$, $\int_\Omega \rho_0 = M_0 > 0$; $(\rho\u)_0 \in L^1(\Omega)$, $(\rho\u)_0 = 0$ almost everywhere on the set $\{x \in \Omega, \rho_0(x) = 0\}$; $\theta_0 > 0$ a.e. in $\Omega$ is such that $\rho_0 s(\rho_0,\theta_0) \in L^1(\Omega)$ and the initial energy of the system satisfies
$$\int_{\Omega} \left(\frac{1}{2\rho_0} |(\rho\u)_0|^2 + \rho_0 e (\rho_0, \theta_0)\right)dx < \infty.$$
Assume that the structural hypotheses \eqref{const-stress}--\eqref{heatflux-cond}, \eqref{entropyproduction}, \eqref{gibbs}--\eqref{S-cond} hold. Then for any $T > 0$ the Navier-Stokes-Fourier system admits a weak solution $(\rho,\u,\theta)$ on $(0,T)\times\Omega$ in the sense specified above.
\end{theorem}

Moreover, the authors also proved that every weak solution satisfies the total dissipation balance \cite[Equation (2.51)]{Fei2}, i.e.,

\begin{multline}\label{totaldiss}
\int_\Omega\left(\frac 12 \rho |\u|^2 + H^{\overline\theta}(\rho, \theta)\right)(\tau) dx + \overline\theta \sigma\left[[0,\tau]\times \overline \Omega\right] \\
=  \int_\Omega \left(\frac 12 \rho_0|\u_0|^2 + H^{\overline\theta}(\rho_0,\theta_0)\right)dx,
\end{multline}
for almost all $\tau \in [0,T]$ and for every positive constant $\overline\theta$.

\begin{remark}\label{LipBoundary}
It may seem that  Theorem \ref{t:weaksolex} is not suitable for our problem since the domain under consideration is of the form $(a,b)\times(c,d)\times(0,1)$ and hence not of class $C^{2,\nu}$. We overcome this issue by the following consideration. The smoothness of the domain is used in the proof of the existence of weak solutions  to ensure the smoothness of  Galerkin approximations. In our case we may use the special structure of the spatial domain together with the boundary conditions to extend any solution from $\Omega$ appropriately (as an even or an odd function) to create a solution with periodic boundary conditions on a larger box where no restrictions on the smoothness of the boundary are necessary.
\end{remark}

\subsection{Classical solutions to the $1$D-system} \label{classicsol-section}

As $\ep \to 0$ we observe that $\Omega_\ep \to (0,1)$. Moreover, we expect the solutions $(\rho_\ep, \u_\ep,  \theta_\ep)$ of  \eqref{continuityeq-main}--\eqref{entropyeq-main}  on $\Omega_\ep$ to converge to a classical solution $( \tilde{\rho}, \tilde{u}, \tilde{\theta})$ of \eqref{continuityeq-1D}--\eqref{entropyeq-1D}. The boundary conditions \eqref{eq:BCq} and \eqref{eq:BC} naturally lead to the no-slip boundary conditions for the velocity and the heat flux, i.e., 
\begin{equation}\label{bc1D}
\tilde{u}(\cdot, 0) = \tilde{u}(\cdot, 1) = 0, \ \ \  q(\tilde{\theta}, \de_y\tilde{\theta})(\cdot, 0) = q(\tilde{\theta}, \de_y\tilde{\theta})(\cdot, 1) = 0.
\end{equation}

There  has been published a lot of papers about the one-dimensional Navier-Stokes-Fourier system \eqref{continuityeq-1D}--\eqref{entropyeq-1D} with the no-slip  boundary conditions. We refer the reader to \cite {AnKaMo}, \cite{Kawo} and \cite{Valli} where the existence of solutions was provided under more restrictive assumptions on the pressure,  viscosity,  energy, etc. However, as far as we know, the global in time existence result for the system in its full generality has not been proven yet. Nevertheless, the local in time existence for any smooth initial data or global in time existence for small data can be expected to hold from the classical results on the topic. Since the existence of a classical one-dimensional solution is not a goal of this paper, we simply assume its existence without any proof. Let $T>0$. We assume that there exists a trio 

$$( \tilde \rho, \tilde u, \tilde \theta):[0,T]\times [0,1] \mapsto (0,\infty) \times \mathbb{R} \times (0,\infty),$$ of smooth functions that is the solution to \eqref{continuityeq-1D}-\eqref{entropyeq-1D}  and  \eqref{bc1D} on $[0,T]\times (0,1)$ satisfying 
$$\tilde \rho \geq c>0, \ \ \ \tilde \theta \geq c>0,$$
 with the initial conditions \eqref{1DIC}. 

%\todo{the limit passage asks for some minimal regularity on the solution, does it ask something more on initial data then what comes from weak solution existence theory? should we include the minimal regularity function spaces? the initial data should be specified at least such so the weak solution exists}

\section{Main Result}

The main result of this paper reads as follows.

\begin{theorem} \label{mainthm}
Let $Q \subset \R^2$ be an open rectangular domain and let $\Omega_\ep = \ep Q \times (0,1)$ for $\varepsilon >0$. Suppose that the structural hypotheses  \eqref{const-stress}--\eqref{heatflux-cond} and \eqref{entropyproduction} for the viscous stress tensor $\S$ and the heat flux $\q$ are satisfied together with \eqref{gibbs}--\eqref{S-cond} for the thermodynamic functions $p$, $e$ and $s$.

Assume that the initial data $(\rho_{0,\ep}, (\rho\u)_{0,\ep}, \theta_{0,\ep})$  satisfy all the assumptions of Theorem \ref{t:weaksolex} on domains $\Omega_\ep$ and denote $(\rho_\ep,\u_\ep,\theta_\ep)$ the corresponding sequence of weak solutions to the 3D Navier-Stokes-Fourier system on $(0,T)\times\Omega_\ep$ emanating from the initial data $(\rho_{0,\ep}, (\rho\u)_{0,\ep}, \theta_{0,\ep})$.

Let $(\tilde{\rho}_0, \tilde{u}_0, \tilde{\theta}_0)$ be smooth  functions such that there exists the classical solution $(\tilde{\rho}, \tilde{u}, \tilde{\theta})$ to the 1D Navier-Stokes-Fourier system on $(0,T)\times(0,1)$ emanating from $(\tilde{\rho}_0, \tilde{u}_0, \tilde{\theta}_0)$. Define $\tilde{\u}_0 = [0,0,\tilde{u}_0]$ and $\tilde{\u} = [0,0,\tilde{u}]$.

Let moreover
\begin{equation}\label{ititiallimit}
\begin{split}
\frac1{|Q_\ep|} \int_{Q_\ep} \rho_{0,\ep} (x_h, \cdot )dx_h \to \tilde{\rho}_0,\ \ \ \frac1{|Q_\ep|} \int_{Q_\ep} (\rho\u)_{0,\ep} (x_h, \cdot )dx_h \to \tilde{\rho}_0\tilde{\u}_0,\\
\frac1{|Q_\ep|} \int_{Q_\ep} \rho_{0,\ep} s(\rho_{\ep,0},\theta_{\ep,0}) dx_h \to \tilde \rho_0s(\tilde\rho_0,\tilde \theta_0),
\end{split}
\end{equation}
weakly in $L^1(0,1)$ and let 
 
\begin{multline}
\label{initlim2}\frac1{|Q_\ep|} \int_{\Omega_\ep} \left[\frac1{2\rho_{0,\ep}} |(\rho\u)_{0,\ep}|^2  + \rho_{\ep,0}e(\rho_{\ep,0},\theta_{\ep,0})\right] dx \\ \to  \int_0^1\left[ \frac12\tilde{\rho}_0 |\tilde{u}_0|^2 +  \tilde \rho_0e(\tilde \rho_0,\tilde\theta_0)\right] dy.\end{multline}

Then 
$$
\mbox{esssup}_{t\in (0,T)} \frac{1}{|Q_\varepsilon|}\|\rho_\varepsilon - \tilde \rho\|^{\frac 53}_{L^{\frac 53}(\Omega_\varepsilon)}\to 0,
$$
$$
\mbox{esssup}_{t\in (0,T)} \frac{1}{|Q_\varepsilon|}\|\theta_\varepsilon - \tilde \theta \|^2_{L^2(\Omega_\varepsilon)} \to 0,
$$
and
$$
\frac{1}{|Q_\varepsilon|}\|{\bf u}_\varepsilon - \tilde{\bf u}\|^r_{L^r((0,T) \times\Omega_\varepsilon)}\to 0,
$$
for every $r\in [1,2)$ as $\ep \to 0$. 

%$${\rm ess} \sup_{\tau\in (0,T)} \frac1{|Q_\ep|}\int_0^1\int_{Q_\ep} \left(\frac 12 \rho_\ep|\u_\ep - \tilde \u|^2 + \EE(\rho_\ep,\theta_\ep|\tilde \rho,\tilde \theta)\right)dx_hdy\to 0,
%$$
%\todo{shouldnt we write it more humanly? without $\EE$?}

\end{theorem}

%\begin{remark}
%Due to estimates of $\mathcal E$ (see Section \ref{precal}), we may immediately derive that
%$$
%\mbox{esssup}_{t\in (0,T)} \frac{1}{|Q_\varepsilon|}\|\rho_\varepsilon - \tilde \rho\|^{\frac 53}_{L^{\frac 53}(\Omega_\varepsilon)}\to 0,
%$$
%$$
%\mbox{esssup}_{t\in (0,T)} \frac{1}{|Q_\varepsilon|}\|\theta_\varepsilon - \tilde \theta \|^2_{L^2(\Omega_\varepsilon)} \to 0
%$$
%and
%$$
%\frac{1}{|Q_\varepsilon|}\|{\bf u}_\varepsilon - \tilde{\bf u}\|^r_{L^r((0,T) \times\Omega_\varepsilon)}\to 0
%$$
%for every $r\in [1,2)$.
%\end{remark}

\begin{remark} \label{Kornlikerem}
Our result can be viewed as an extension of the dimension reduction  for the barotropic Navier-Stokes system achieved in \cite{BeFeNo} to the full Navier-Stokes-Fourier system. The basic strategy of using the relative entropy inequality is the same. However, the presence of the temperature raises new obstacles.

In elasticity theory, the analysis of dimension reduction problems depends on the use of Korn's inequality which controls the gradient of velocity by its symmetric part, i.e.,

$$\|\nabla_x \ve\|_{L^2(\Omega_\ep)} \leq c(\ep) \|\nabla_x \ve + \nabla_x^T \ve\|_{L^2(\Omega_\ep)}, \ \ \ \ve\cdot \n|_{\partial\Omega_\ep} =0.$$ 
There are two problems that arise with respect to above inequality. Firstly, validity  even for a fixed $\ep>0$ requires certain restrictions on the shape of the cross-section $Q$. Secondly, even "properly" shaped $Q$ might not stop the constant $c(\ep)$ from blowing up as $\ep \to 0$. 

In \cite{BeFeNo} authors obtain their result for a regular planar domain $Q$ since they avoid the use of Korn's  inequality by exploring the structural stability of the family of solutions of the barotropic Navier-Stokes system. In our case, the approach of \cite{BeFeNo} is disrupted by the temperature. 

Therefore our result leans on the validity of stronger Korn's like inequality appropriate for  compressible fluids, namely, we use   

$$
\|\nabla_x \ve\|^2_{L^2(\Omega_\ep)}\leq \left\|\nabla_x \ve + \nabla_x^T \ve - \frac 23 \div_x \ve \mathbb{I}\right\|^2_{L^2(\Omega_\ep)},  \ \ \ \ve\cdot \n|_{\partial\Omega_\ep} =0.
$$
To get that we assume that  $Q$ is an open  rectangular domain in $\mathbb R^2$, i.e.,
$Q$ can be written as 

$$Q = (a,b)\times (c,d), \ \ \ a<b, \ c < d, \ a,b,c,d \in \mathbb R.$$
\end{remark}

\section{Preliminary calculations}
\label{precal}
In this section we introduce  the estimates which will be used repeatedly in the subsequent calculations. Hereinafter, we assume that $(\tilde \rho, \tilde u, \tilde \theta)$ and $(\rho_\ep,\u_\ep,\theta_\ep)$ are a classical and a weak solution to the respective problem satisfying the assumptions introduced in Theorem \ref{mainthm}. For clarity, we omit the suffix $\ep$ where no confusion occurs and we write $(\rho,\u,\theta)$ instead of $(\rho_\ep,\u_\ep,\theta_\ep)$.

Following \cite[Chapters 4,5]{Fei2} we introduce essential and residual components based on  $\rho$ and $\theta$. To begin, we choose positive constants $\underline \rho, \overline \rho, \underline \theta, \overline \theta$ fulfilling

$$ 0 < \underline \rho \leq \frac12 \underset{(\tau,y) \in [0,T ]\times [0,1]}\min \tilde \rho (\tau,y) \leq 2 \underset{(\tau,y) \in [0,T ]\times [0,1]}\max \tilde \rho (\tau,y) \leq \overline \rho,$$ 

$$ 0 < \underline \theta \leq \frac12 \underset{(\tau,y) \in [0,T ]\times [0,1]}\min \tilde \theta (\tau,y) \leq 2 \underset{(\tau,y) \in [0,T ]\times [0,1]}\max \tilde \theta (\tau,y) \leq \overline \theta.$$ 

According to Lemma 5.1 in \cite{Fei2} there exists a constant $c>0$ such that 

\eqn{entfuncest} 
$$ \EE(\rho, \theta| \tilde \rho, \tilde \theta) \geq c \left\{ \begin{array}{ll} |\rho - \tilde \rho|^2 + |\theta - \tilde \theta|^2 & \mbox{ if } (\rho, \theta) \in [\underline{\rho}, \overline{\rho}]\times [\underline{\theta}, \overline{\theta}], \\ \\
1 + |\rho s (\rho, \theta)| + \rho e(\rho, \theta) & \mbox{ otherwise}. \end{array} \right.$$
It is worth pointing out that $c$ is independent of $\ep$.

Each measurable function $h$ can be written as 

$$h = h_{ess} + h_{res},$$
where

$$h_{ess}(t,x) = \left\{ \begin{array}{cl} h(t,x) &\mbox{ if } (\rho(t,x), \theta(t,x)) \in [\underline{\rho}, \overline{\rho}]\times [\underline{\theta},\overline{\theta}],\\ \\ 0 &\mbox{ otherwise}. \end{array}\right. %\ \ \ h_{res} = h - h_{ess}.
$$

\subsection{Estimates on $\rho$ and $\theta$}

We immediately see from \eqref{entfuncest} that 

\begin{equation} \label{essest} \| [\theta - \tilde \theta]_{ess}\|_{L^s(\Omega_\varepsilon)}^s  + \| [\rho - \tilde \rho]_{ess} \|_{L^s(\Omega_\varepsilon)}^s 
\leq c\int_{\Omega_\varepsilon} \EE(\rho, \theta|\tilde \rho, \tilde \theta) dx,
\end{equation}
for $s\geq 2$.
We show that similar estimates hold also for the residual parts. Firstly,

\begin{equation*}
 \rho e(\rho, \theta) \geq c (\rho^\frac53 +  \theta^4).
\end{equation*}
Since $\displaystyle\frac{P(z)}{z^\frac53}$ is decreasing and $\displaystyle \lim_{z\to \infty} \frac{P(z)}{z^\frac53} = P_\infty$ we obtain $$P(z) \geq P_\infty z^\frac53.$$
Using the above estimate we get

$$ \rho e(\rho, \theta) \geq \frac32 \theta^\frac52 P_\infty \left(\frac{\rho}{\theta^\frac32}\right)^\frac53 +  a \theta^4 \geq c(  \rho^\frac53 + \theta^4).$$  
Now we can estimate the residual parts. We have for  $1\leq  q \leq 4$ that

$$ \int_{\Omega_\varepsilon} | [\theta -\tilde \theta ]_{res} |^q  dx   \leq  c \int_{\Omega_\varepsilon} (| [\theta ]_{res} |^q +1_{res} ) dx ,$$
and by the use of H\" older and Young inequalities together with \eqref{entfuncest} we obtain

$$\leq c \left(\int_{\Omega_\varepsilon} | [\theta ]_{res} |^4 dx\right)^\frac{q}4 \left(\int_{\Omega_\varepsilon}  1_{res}  dx\right)^{1-\frac{q}4} + c\int_{\Omega_\varepsilon} \EE(\rho, \theta|\tilde \rho, \tilde \theta)  dx $$

$$\leq c   \int_{\Omega_\varepsilon} | [\theta ]_{res} |^4 dx  + c\int_{\Omega_\varepsilon} \EE(\rho, \theta|\tilde \rho, \tilde \theta).$$
Thus 
\begin{equation}
\label{resestheta} \int_{\Omega_\varepsilon} | [\theta -\tilde \theta ]_{res} |^q  dx  \leq  c\int_{\Omega_\varepsilon} \EE(\rho, \theta|\tilde \rho, \tilde \theta) dx ,
\end{equation}
where $1 \leq q \leq 4$. 
Combining \eqref{essest} and \eqref{resestheta} we obtain 

\begin{equation}
\label{esthetanon} \| \theta -\tilde \theta  \|_{L^2(\Omega_\ep)}^2    \leq  c\int_{\Omega_\varepsilon} \EE(\rho, \theta|\tilde \rho, \tilde \theta) dx  .
\end{equation}

As for the density, we get analogously as above that for  $1 \leq p \leq \frac53$ it holds 

$$   \int_{\Omega_\ep} | [\rho - \tilde \rho]_{res} |^p dx  \leq c \int_{\Omega_\varepsilon} \left(|[ \rho ]_{res}|^p  + 1_{res}\right)dx  $$

$$\leq c \left(\int_{\Omega_\varepsilon} | [\rho ]_{res} |^\frac53 dx\right)^\frac{3p}5 \left(\int_{\Omega_\varepsilon} 1_{res}  dx\right)^{1-\frac{3p}5} +c \int_{\Omega_\varepsilon} \EE(\rho, \theta|\tilde \rho, \tilde \theta) dx $$

$$ \leq c  \int_{\Omega_\varepsilon}| [\rho ]_{res} |^\frac53 dx  + c \int_{\Omega_\varepsilon} \EE(\rho, \theta|\tilde \rho, \tilde \theta) dx.$$
Hence, we have

\eqn{resestrho} $$ 
\int_{\Omega_\ep} | [\rho - \tilde \rho]_{res} |^p dx  \leq c \int_{\Omega_\varepsilon} \EE(\rho, \theta|\tilde \rho, \tilde \theta) dx,
$$
for $1 \leq p \leq \frac53$.

\subsection{Korn and Poincar\'e inequalities}
Although both inequalities are very well known, we need the estimates which are independent of $\ep$. Note that this goal cannot be reached by a simple rescaling argument as the domain shrinks only in two dimensions. For more details see Remark \ref{Kornlikerem}.

For clarity,  in this section we prefer  the notation $x_3$ and $\partial_{ x_3}$ instead of $y$ and $\partial_y$. 

\begin{lemma}\label{kornlike}
Let $\Omega\subset \mathbb R^3$ be a rectangular domain, i.e., $\Omega = (a_1,b_1)\times (a_2,b_2)\times (a_3,b_3)$ and  $\u\in W^{1,2}(\Omega,\mathbb R^3)$ be such that $\u\cdot\n = 0$ on $\partial \Omega$. Then
$$
\|\nabla_x \u\|^2_{L^2(\Omega)}\leq \left\|\nabla_x \u + \nabla_x^T \u\right\|^2_{L^2(\Omega)},
$$

$$
\|\nabla_x \u\|^2_{L^2(\Omega)}\leq \left\|\nabla_x \u + \nabla_x^T \u - \frac 23 \div_x \u \mathbb{I}\right\|^2_{L^2(\Omega)},
$$

$$\|\nabla_x \u\|^2_{L^2(\Omega)} \leq \int_\Omega (\nabla_x \u + \nabla_x^T  \u - \frac23 \div_x \u \mathbb{I}):\nabla_x \u dx.$$
\end{lemma}
\begin{proof}
Since smooth functions are dense in $W^{1,2}(\Omega,\mathbb R^3)$, we prove the lemma only for $\u \in C^2(\Omega)$. Denote $\u = [u^1,u^2,u^3]$.
We split  $\partial \Omega$ into three parts as follows:
$$\partial \Omega_1 = \{a_1\}\times[a_2,b_2]\times [a_3,b_3] \cup \{b_1\}\times[a_2,b_2]\times [a_3,b_3],$$

$$\partial \Omega_2 = [a_1,b_1]\times\{a_2\}\times  [a_3,b_3] \cup [a_1,b_1]\times\{b_2\}\times [a_3,b_3],$$

$$\partial \Omega_3 = [a_1,b_1]\times [a_2,b_2]\times\{a_3\} \cup [a_1,b_1]\times [a_2,b_2]\times\{b_3\}.$$

Therefore,  boundary conditions on $\u$ imply:

\begin{equation}u^1 |_{\partial\Omega_1} =0, \ \ \ u^2 |_{\partial\Omega_2} =0, \ \ \ u^3 |_{\partial\Omega_3} =0.\end{equation}

Since 

$$ (\nabla_x \u + \nabla_x^T  \u) =\left(\begin{array}{ccc} 2\partial_{x_1} u^1 & \partial_{x_2}u^1 + \partial_{x_1} u^2 & \partial_{x_3} u^1 + \partial_{x_1}u^3 \\
 \partial_{x_1} u^2 + \partial_{x_2}u^1 & 2\partial_{x_2}u^2 & \partial_{x_3}u^2 + \partial_{x_2} u^3 \\
\partial_{x_1}u^3 + \partial_{x_3} u^1& \partial_{x_2}u^3 + \partial_{x_3} u^2& 2 \partial_{x_3}u^3\end{array}\right),$$
we calculate

$$(\nabla_x \u + \nabla_x^T  \u):(\nabla_x \u + \nabla_x^T  \u) =|\nabla_x \u|^2 + 3\left[(\partial_{x_1} u^1)^2 + (\partial_{x_2} u^2)^2 + (\partial_{x_3} u^3)^2\right]$$

$$+ 4 \left[ \partial_{x_2} u^1\partial_{x_1} u^2 + \partial_{x_3} u^1\partial_{x_1} u^3+ \partial_{x_3} u^2\partial_{x_2} u^3\right].$$
We integrate by parts to modify the last terms.

First, we integrate by parts in $x_2$,

$$\int_\Omega \partial_{x_2} u^1\partial_{x_1} u^2  dx = \int_{a_1}^{b_1} \int_{a_3}^{b_3}\left[u^1 \partial_{x_1}u^2\right]_{a_2}^{b_2}dx_3  dx_1- \int_\Omega u^1 \partial_{x_2} \partial_{x_1} u^2 dx.$$
 From the  boundary condition $u^2|_{\partial\Omega_2} = 0$ we get $\partial_{x_1} u^2|_{\partial\Omega_2} = 0$ and hence the boundary term above disappears. Next, we integrate by parts in $x_1$,

$$ - \int_\Omega u^1  \partial_{x_1}\partial_{x_2}  u^2 dx  = \int_{a_2}^{b_2} \int_{a_3}^{b_3}\left[u^1 \partial_{x_2}u^2\right]_{a_1}^{b_1}dx_3 dx_2 +
\int_\Omega\partial_{x_1} u^1 \partial_{x_2}  u^2 dx.$$ 
The boundary term above disappears due to $u^1 |_{\partial\Omega_1} =0$ and  we end up with 

$$\int_\Omega \partial_{x_2} u^1\partial_{x_1} u^2  dx  =\int_\Omega \partial_{x_1} u^1 \partial_{x_2}  u^2 dx.$$

The rest of the terms is treated analogously and we get 

$$\int_\Omega \partial_{x_3} u^1\partial_{x_1} u^3  dx  =\int_\Omega \partial_{x_1} u^1 \partial_{x_3}  u^3 dx,$$
and

$$\int_\Omega \partial_{x_3} u^2\partial_{x_2} u^3  dx  =\int_\Omega \partial_{x_2} u^2 \partial_{x_3}  u^3 dx.$$

Finally,

\begin{align*}
&\int_\Omega (\nabla_x \u + \nabla_x^T  \u):(\nabla_x \u + \nabla_x^T  \u)dx = \int_\Omega|\nabla_x \u|^2dx \\
& \quad + 3\int_\Omega \left[(\partial_{x_1} u^1)^2 + (\partial_{x_2} u^2)^2 + (\partial_{x_3} u^3)^2\right]dx \\
& \quad + 4 \int_\Omega  \left[ \partial_{x_1} u^1\partial_{x_2} u^2 + \partial_{x_1} u^1\partial_{x_3} u^3+ \partial_{x_2} u^2\partial_{x_3} u^3\right]dx,
\end{align*}
and since
$$ (\div_x \u)^2 = \sum_{i=1}^3 (\partial_{x_i} u^i)^2 + 2[\partial_{x_1} u^1\partial_{x_2} u^2 + \partial_{x_1} u^1\partial_{x_3} u^3+ \partial_{x_2} u^2\partial_{x_3} u^3],$$
we get that 

$$\int_\Omega (\nabla_x \u + \nabla_x^T  \u):(\nabla_x \u + \nabla_x^T  \u)dx = \int_\Omega|\nabla_x \u|^2dx + 2 \int_\Omega (\div_x \u)^2 dx $$

$$+  \int_\Omega \left[(\partial_{x_1} u^1)^2 + (\partial_{x_2} u^2)^2 + (\partial_{x_3} u^3)^2\right]dx .$$
Thus the first inequality is proven. The second inequality follows easily since 

\begin{align*}
&(\nabla_x \u + \nabla_x^T  \u - \frac23 \div_x \u \mathbb{I}):(\nabla_x \u + \nabla_x^T  \u - \frac23 \div_x \u \mathbb{I}) \\
&\quad = (\nabla_x \u + \nabla_x^T  \u):(\nabla_x \u + \nabla_x^T  \u) - \frac43 (\div_x \u)^2. 
\end{align*}

We use the integration by parts in the same way as before in order to get

$$\int_\Omega (\nabla_x \u + \nabla_x^T  \u - \frac23 \div_x \u \mathbb{I}):\nabla_x \u dx = \int_\Omega |\nabla_x \u|^2dx +\frac13 \int_\Omega(\div_x \u)^2 dx,$$
which directly implies the last desired inequality.
\end{proof}

We introduce the following notation. For a set $M\subset \mathbb R^d$, $d\in \mathbb N$ and a function $f\in L^1(M)$ we denote by $(f)_M$ its integral average, i.e., $$(f)_M = \frac 1{|M|}\int_M f(x) dx.$$

\begin{lemma}\label{nonepoin}
There exists a constant $c>0$ independent of $\ep$ such that 
 for every $f\in W^{1,2}(\Omega_\varepsilon)$ fulfilling $f(\cdot, 0) = f(\cdot, 1)  = 0$ on $Q_\ep$ it holds that

$$
\int_0^1 \int_{Q_\ep} |f(x_h,y)-(f)_{Q_\ep}(y)|^4dx_h dy = \|f- (f)_{Q_\ep}\|^4_{L^4(\Omega_\ep)}\leq c \|\nabla f\|^4_{L^2(\Omega_\ep)}.
$$

\end{lemma}

\begin{proof}
From Ladyzhenskaya's inequality (see \cite{ladyzhenskaya}) we have 

$$
\|f-(f)_{Q_\ep}\|^4_{L^4(Q_\varepsilon)}\leq c \|\nabla_{x_h} f\|_{L^2(Q_\varepsilon)}^2 \|f-(f)_{Q_\ep}\|_{L^2(Q_\varepsilon)}^2.
$$
with the constant $c$ independent of $\varepsilon$. Indeed, as this is a 2D inequality, the independence can be shown by a simple rescaling argument. Further, 

\begin{align*}
&\|f-(f)_{Q_\ep}\|_{L^4(\Omega_\ep)}^4 = \int_0^1 \|f-(f)_{Q_\ep}\|_{L^4(Q_\varepsilon)}^4dy \\
&\quad \leq c \int_0^1 \|\nabla_{x_h} f\|_{L^2(Q_\varepsilon)}^2\|f-(f)_{Q_\ep}\|_{L^2(Q_\varepsilon)}^2dy \\ 
&\quad \leq c\|f-(f)_{Q_\ep}\|^2_{L^\infty((0,1),L^2(Q_\varepsilon))} \|\nabla f\|_{L^2(\Omega_\varepsilon)}^2\leq c \|\nabla f\|_{L^2(\Omega_\varepsilon)}^4,
\end{align*}
where we used the fact that, due to H\"older inequality,

\begin{multline*}
\|f-(f)_{Q_\ep}\|_{L^2(Q_\ep)}^2 \leq \int_{Q_\ep} \left|f(x_h) - \frac 1{|Q_\ep|} \int_{Q_\ep}f(z_h)dz_h\right|^2 dx_h \\
\leq \int_{Q_\ep} \frac 1{|Q_\ep|} \int_{Q_\ep} |f(x_h) - f(z_h)|^2dz_h dx_h\\
\leq c\frac 1{|Q_\ep|} \left(\int_{Q_\ep} \int_{Q_\ep} |f(x_h)|^2dx_hdz_h  + \int_{Q_\ep} \int_{Q_\ep} |f(z_h)|^2dx_hdz_h \right)\\
 \leq c\|f\|_{L^2(Q_\ep)}^2,
\end{multline*}
for a.a. $y\in (0,1)$ and thus by Sobolev-Poincar\'e inequality we finally obtain

\begin{multline*}
\|f-(f)_{Q_\ep}\|^2_{L^\infty((0,1),L^2(Q_\varepsilon))}\leq c \|f\|_{L^\infty((0,1),L^2{(Q_\ep)})}^2\\ \leq c\|\partial_{x_3} f\|_{L^2((0,1),L^2(Q_\ep))}^2\leq c \|\nabla f\|_{L^2(\Omega_\ep)}^2.
\end{multline*}
\end{proof}

\section{Relative entropy balance}

Recall that we omit the index $\ep$ for the functions $(\rho_\ep,\u_\ep,\theta_\ep)$ and simply write $(\rho,\u,\theta)$ instead. Following the calculations in \cite[Section 3]{Fei1} we obtain for each $\ep >0$ the following relative entropy inequality satisfied by any weak solution $(\rho, \u,\theta)$ to the Navier-Stokes-Fourier system on $\Omega_\ep$ and any trio $(r,\U, \Theta)$ of smooth functions, $r$ and $\Theta$ bounded below away from zero in $[0,T]\times \Omega_\ep$ and $\U\cdot \n|_{\partial \Omega_{\ep}}=0$.

$$\frac1{|Q_\ep|} \int_{\Omega_\ep} \left(\frac12 \rho | \u - \U|^2 + \EE(\rho, \theta|r,\Theta)\right)(\tau, \cdot )dx $$

$$+ \frac1{|Q_\ep|} \int_0^\tau \int_{\Omega\ep} \frac{\Theta}{\theta}\left(\S(\theta, \nabla_x \u):\nabla_x \u - \frac{\q(\theta, \nabla_x \theta)\cdot \nabla_x \theta}{\theta}\right)dx dt $$

$$\leq\frac1{|Q_\ep|}\int_{\Omega_\ep} \left(\frac1{2\rho_{0,\ep}}|(\rho\u)_{0,\ep} - \rho_{0,\ep}\U(0,\cdot)|^2 + \EE(\rho_{\ep,0}, \theta_{\ep,0}|r(0,\cdot), \Theta(0,\cdot))\right) dx$$

$$+ \frac1{|Q_\ep|}\int_0^\tau \int_{\Omega_\ep} \rho(\u - \U)\cdot \nabla_x \U \cdot(\U-\u)dx dt$$

$$ + \frac1{|Q_\ep|}\int_0^\tau \int_{\Omega_\ep} \rho (s(\rho, \theta) - s(r, \Theta))(\U - \u)\cdot \nabla_x \Theta dx dt $$

$$+\frac1{|Q_\ep|}\int_0^\tau \int_{\Omega_\ep} (\rho (\de_t \U + \U \cdot \nabla_x \U)\cdot (\U-\u) - p(\rho, \theta)\div_x \U + \S (\theta, \nabla_x \u):\nabla_x \U)dx dt$$

$$-\frac1{|Q_\ep|}\int_0^\tau \int_{\Omega_\ep} \left(\rho(s (\rho, \theta) - s (r,\Theta))(\de_t \Theta +\U \cdot \nabla_x\Theta ) + \frac{\q (\theta, \nabla_x\theta)}{\theta} \cdot \nabla_x \Theta \right) dxdt$$

\eqn{relentineq} 
$$ + \frac1{|Q_\ep|}\int_0^\tau \int_{\Omega_\ep} \left(\left(1-\frac\rho{r}\right)\de_t p(r, \Theta) - \frac\rho{r} \u \cdot \nabla_x p (r, \Theta)\right)dx dt,
$$
for almost all $\tau \in [0,T]$, where $\EE$ was introduced in \eqref{entfunc}. 

To prove Theorem \ref{mainthm} we take   
$$r = \tilde \rho(t,y), \ \ \ \Theta = \tilde \theta(t,y), \ \ \ \U  = \tilde \u(t,y) = \left[\begin{array}{c} 0 \\ 0 \\ \tilde u(t,y)\end{array}\right],$$
in  \eqref{relentineq} to obtain 

$$ \frac1{|Q_\ep|}\int_{\Omega_\ep} \left(\frac12 \rho | \u - \tilde \u|^2 + \EE(\rho, \theta|\tilde \rho,\tilde\theta)\right)(\tau , \cdot )dx $$

$$+\frac1{|Q_\ep|} \int_0^\tau  \int_{\Omega\ep} \frac{\tilde \theta}{\theta}\left(\S(\theta, \nabla_x \u):\nabla_x \u - \frac{\q(\theta, \nabla_x \theta)\cdot \nabla_x \theta}{\theta}\right)dx dt$$

$$\leq \frac1{|Q_\ep|} \int_{\Omega_\ep} \left(\frac1{2\rho_{0,\ep}}|(\rho\u)_{0,\ep} - \rho_{0,\ep}\tilde{\u}_0|^2 + \EE(\rho_{0,\ep}, \theta_{0,\ep}|\tilde{\rho}_0, \tilde{\theta}_0)\right) dx$$

$$+ \frac1{|Q_\ep|} \int_0^\tau  \int_{\Omega_\ep} \rho|u^3 - \tilde u|^2| \de_y \tilde u |dx dt + \frac1{|Q_\ep|}\int_0^\tau  \int_{\Omega_\ep} \rho (s(\rho, \theta) - s(\tilde \rho, \tilde \theta))(\tilde u -  u^3)\cdot \partial_y \tilde \theta dx dt $$

$$+\frac1{|Q_\ep|}\int_0^\tau  \int_{\Omega_\ep} \left(\rho (\de_t \tilde u + \tilde u  \partial_y \tilde u) (\tilde u-u^3) - p(\rho,  \theta)\partial_y \tilde u + \S (\theta, \nabla_x \u):\nabla_x \tilde \u\right)dx dt$$

$$-\frac1{|Q_\ep|}\int_0^\tau  \int_{\Omega_\ep} \left(\rho(s (\rho,  \theta) - s (\tilde \rho,\tilde \theta))(\de_t \tilde \theta + \tilde u \partial_y \tilde \theta)+ \frac{\q (\theta, \nabla_x \theta)}{\theta} \cdot \nabla_x \tilde \theta \right) dxdt$$

\eqn{relentineq2} $$+ \frac1{|Q_\ep|}\int_0^\tau  \int_{\Omega_\ep} \left(\left(1-\frac{\rho}{\tilde \rho}\right)\de_t p(\tilde \rho, \tilde \theta) - \frac\rho{\tilde \rho} \u \cdot \nabla_x p (\tilde \rho, \tilde \theta)\right)dx dt,
$$
for almost all $\tau \in [0,T]$.

In order to handle the integrals on the right-hand side of \eqref{relentineq2} we proceed in several steps:

\bigskip
\noindent{\it Step 0}  
Observe that by \eqref{ititiallimit} and \eqref{initlim2} {we get} 

\eqn{step0finaleq}
$$\frac1{|Q_\ep|} \int_{\Omega_\ep} \left(\frac1{2\rho_{0,\ep}}|(\rho\u)_{0,\ep} - \rho_{0,\ep}\tilde{\u}_0|^2 + \EE(\rho_{0,\ep}, \theta_{0,\ep}|\tilde{\rho}_0, \tilde{\theta}_0)\right) dx \to 0,$$
as $\ep \to 0$. From now on, we include this term in $\Gamma(\ep)$ where $\Gamma(\ep) \to 0$ as $\ep \to 0$.

\bigskip
\noindent{\it Step 1}   

\eqn{step1finaleq} $$
 \int_{\Omega_\ep} \rho|u^3 - \tilde u|^2| \de_y \tilde  u  |dx \leq  2 \|\de_y   \tilde u\|_{L^\infty(0,1)}   \int_{\Omega_\ep} \frac12 \rho |u^3 - \tilde u |^2 dx. 
$$

\bigskip
\noindent{\it Step 2} 

$$\left|\int_{\Omega_\varepsilon} \rho \left(s(\rho, \theta) - s (\tilde \rho, \tilde \theta)\right) (\tilde u - u^3)\partial_y \tilde \theta dx\right|$$

$$\leq \|\de_y \tilde \theta \|_{L^\infty(0,1)} \left[\overline \rho \int_{\Omega_\ep} \left|\left[ s (\rho, \theta)  - s (\tilde \rho, \tilde \theta )\right]_{ess}\right||u^3 - \tilde u|dx \right.$$

$$\left. +  \int_{\Omega_\ep} \left|\left[ \rho\left(s (\rho, \theta)  - s (\tilde \rho, \tilde \theta )\right)\right]_{res}\right||u^3 - \tilde u|dx\right].$$

First, we estimate the essential part. Since
$$\left|\left[s (\rho, \theta) - s (\tilde \rho, \tilde \theta)\right]_{ess} \right|= \left| \left[S\left(\frac{ \rho}{\theta^\frac32}\right) +  \frac{4a}3 \frac{\theta^3}\rho - S\left(\frac{\tilde \rho}{\tilde \theta^\frac32}\right) -  \frac{4a}3 \frac{\tilde \theta^3}{\tilde \rho}\right]_{ess}\right|$$

$$\leq \left| \left[ S\left(\frac{\rho}{\theta^\frac32}\right) - S\left(\frac{\tilde \rho}{\tilde \theta^\frac32}\right)\right]_{ess} \right|  +   \frac{4a}3 \left| \left[   \frac{\theta^3}\rho  -  \frac{\tilde \theta^3}{\tilde \rho}\right]_{ess}\right|$$

$$\leq  |S'\left( \xi\right) | \left|\left[\frac{\rho}{\theta^\frac32} - \frac{\tilde \rho}{\tilde \theta^\frac32}\right]_{ess}\right|  +   \frac{4a}3 \left| \left[   \frac{\theta^3}\rho  -  \frac{\tilde \theta^3}{\tilde \rho}\right]_{ess}\right|\leq c|[\rho - \tilde \rho]_{ess}| +  c |[\theta- \tilde \theta]_{ess}|, $$ 
thus,

\begin{multline*}
\int_{\Omega_\ep} \left|\left[ s (\rho, \theta)  - s (\tilde \rho, \tilde \theta )\right]_{ess}\right||u^3 - \tilde u|dx \\
 \leq K(\cdot) \int_{\Omega_\ep} \left(\frac12 \rho |u^3 - \tilde u |^2 + \EE(\rho, \theta|\tilde \rho, \tilde \theta) \right)dx.
\end{multline*}
Here and hereafter, $K(\cdot)$ is a generic constant depending on $\tilde \rho$, $\tilde u$, $\tilde \theta$, $\underline \rho$, $\underline \theta$ through its respective norms. It is independent of $\varepsilon$.

Next we treat the residual part. We compute

\begin{multline*} 
\int_{\Omega_\varepsilon} \left|\left[\rho(s(\rho,\theta) - s(\tilde \rho,\tilde \theta))\right]_{{res}}\right||\tilde u - u^3| dx\\ \leq
\int_{\Omega_\varepsilon} \left|\left[\rho(s(\rho,\theta) - s(\tilde \rho,\tilde \theta))\right]_{{res}}\right||\tilde u - (u^3)_{Q_\ep}| dx\\ + 
\int_{\Omega_\varepsilon} \left|\left[\rho(s(\rho,\theta) - s(\tilde \rho,\tilde \theta))\right]_{{res}}\right||(u^3)_{Q_\ep} - u^3| dx= \mathcal I_1 + \mathcal I_2.
\end{multline*}
It holds

$$
\mathcal I_1 \leq \|\tilde u - (u^3)_{Q_\ep}\|_{L^\infty(\Omega_\varepsilon)}\|[\rho(s(\rho,\theta) - s(\tilde \rho,\tilde\theta))]_{res}\|_{L^{\frac 43}(\Omega_\ep)} |Q_\ep|^{\frac 14}.
$$

Further,

\begin{multline}\label{residualest}
|[\rho(s(\rho,\theta) - s(\tilde \rho,\tilde \theta)]_{res}|\leq c[\rho + \rho s(\rho,\theta)]_{res}\\ \leq c[\rho  + \theta^3 + \rho[\log \theta]^{+} +\rho|\log \rho|]_{res},
\end{multline}
(cf. Equation (3.39) in \cite{Fei2}). Each term on the right hand side of \eqref{residualest} can be estimated using \eqref{resestheta} and \eqref{resestrho} as follows:

$$\left(\int_{\Omega_\ep} \rho_{res}^{\frac 43}dx\right)^{\frac 34}\leq \left(\int_{\Omega_\ep}\rho_{res}^{\frac 53}dx\right)^{\frac 35} |Q_\ep|^{\frac 3{20}} \leq \left(\frac 1{|Q_\ep|}\int \EE(\rho,\theta|\tilde\rho,\tilde \theta)dx\right)^{\frac 35}|Q_\ep|^{\frac 34},$$

$$
\left(\int_{\Omega_\ep} \theta_{res}^4dx\right)^{\frac 34}\leq \left(\frac 1{|Q_\ep|}\int_{\Omega_\ep} \EE(\rho,\theta|\tilde \rho,\tilde \theta)dx\right)^{\frac 34} |Q_\ep|^{\frac 34},
$$

\begin{multline*}
\left(\int_\ep (\rho[\log \theta]_{res}^+)^{\frac 43}dx \right)^{\frac 34} \leq c\left(\int_{\Omega_\ep} \rho_{res}^{\frac 43} \theta_{res}^{\frac 45}dx\right)^{\frac 34}\\ \leq c \left(\int_{\Omega_\ep} \rho_{res}^{\frac 53}dx\right)^{\frac 35}\left(\int_{\Omega_\ep} \theta_{res}^4dx\right)^{\frac 3{20}}\\
\leq \left(\frac 1{|Q_\ep|}\int_{\Omega_\ep}\EE(\rho,\theta|\tilde \rho,\tilde \theta)dx\right)^{\frac 34}|Q_\ep|^{\frac 34},
\end{multline*}

\begin{multline*}
\left(\int_{\Omega_\ep} (|\rho \log \rho|_{res})^{\frac43}dx\right)^{\frac 34} \leq c \left(\int_{\Omega_\ep}( \rho^{\frac 43} + \rho^{\frac 53})dx\right)^{\frac 34}\\  \leq \left(\frac 1{|Q_\ep|}\int_{\Omega_\ep}\EE(\rho,\theta|\tilde \rho,\tilde \theta)dx\right)^{\frac 34}|Q_\ep|^{\frac 34}.
\end{multline*}
From  \eqref{totaldiss} and \eqref{ititiallimit},\eqref{initlim2} we deduce that
$$
\left(\frac 1{|Q_\ep|} \int_{\Omega_\ep} \EE(\rho,\theta|\tilde \rho,\tilde \theta) dx\right)\in L^\infty(0,T),
$$
with a bound independent of $\ep$ (cf. Equation (2.52) in \cite{Fei2}).  Thus

$$\mathcal I_1 \leq \|\tilde u - (u^3)_{Q_\ep}\|_{L^\infty(\Omega_\ep)} \left(\frac 1{|Q_\ep|}\int_{\Omega_\ep} \EE(\rho,\theta|\tilde \rho,\tilde \theta)dx\right)^{\frac 12} |Q_\ep|.$$
As far as $W^{1,1}(0,1)\hookrightarrow L^\infty(0,1)$ we have 

\begin{multline*}
\|\tilde u - (u^3)_{Q_\ep}\|_{L^\infty(\Omega_\ep)} = \|\tilde u - (u^3)_{Q_\ep}\|_{L^\infty(0,1)} \leq  \|\partial_y (\tilde u - (u^3)_{Q_\ep})\|_{L^1(0,1)}\\ \leq \frac 1{|Q_\ep|} \|\partial_y (\tilde u - u^3)\|_{L^1(\Omega_\ep)} \leq \frac1{|Q_\ep|^{\frac 12}} \|\partial_y (\tilde u - u^3)\|_{L^2(\Omega_\ep)}.
\end{multline*}
This implies

\begin{multline*}
\mathcal I_1 \leq \|\partial_y(\tilde u - u^3)\|_{L^2(\Omega_\ep)} \left(\frac 1{|Q_\ep|} \int_{\Omega_\ep} \EE(\rho,\theta|\tilde \rho,\tilde \theta)dx\right)^{\frac 12} |Q_\ep|^{\frac 12} \\ \leq \delta \|\partial_y (\tilde u - u^3)\|_{L^2(\Omega_\ep)}^2 + K(\delta)\int_{\Omega_\ep} \EE(\rho,\theta|\tilde \rho, \tilde \theta)dx.
\end{multline*}

Summarizing the calculations above we get

$$
\|[\rho(s(\rho,\theta) - s(\tilde \rho,\tilde \theta))]_{res}\|_{L^{\frac 43}(\Omega_\ep)} \leq \left(\frac1{|Q_\ep|}\int_{\Omega_\ep} \EE(\rho,\theta|\tilde \rho,\tilde \theta)dx\right)^{\frac 12} |Q_\ep|^{\frac 34},
$$
and thus

\begin{multline*}
\mathcal I_2 \leq \|(u^3)_{Q_\ep} - u^3\|_{L^4(\Omega_\varepsilon)} \left(\int_{\Omega_\ep}\EE(\rho,\theta|\tilde \rho,\tilde \theta)dx\right)^{\frac 12}|Q_\ep|^{\frac 14}\\ \leq c {|Q_\ep|^{\frac 14}}\|\nabla \u\|_{L^2(\Omega_\ep)} \left(\int_{\Omega_\ep} \EE(\rho,\theta|\tilde \rho,\tilde \theta)dx\right)^{\frac 12}\\ \leq  \ep\|\nabla \u\|_{L^2(\Omega_\ep)}^2 + c\int_{\Omega_\ep}\EE(\rho,\theta|\tilde \rho,\tilde \theta)dx,
\end{multline*}
where we used the result of Lemma \ref{nonepoin}.

We conclude  that

\begin{multline}\label{step2finaleq}
\left|\int_{\Omega_\varepsilon} \rho \left(s(\rho, \theta) - s (\tilde \rho, \tilde \theta)\right) (\tilde u - u^3)\partial_y \tilde \theta dx\right|\\
\leq \delta \|\partial_y(\tilde u - u^3)\|_{L^2(\Omega_\ep)}^2 + \ep \|\nabla \u\|_{L^2(\Omega_\ep)}^2\\
 + K(\delta,\cdot) \int_{\Omega_\ep}\left( \frac 12 \rho|u^3 - \tilde u|^2 + \EE(\rho,\theta|\tilde \rho,\tilde\theta)\right)dx,
\end{multline}
for any $\delta >0$. Here and hereafter, $K(\delta, \cdot)$ is a generic constant depending on $\delta$, $\tilde \rho$, $\tilde u$, $\tilde \theta$, $\underline \rho$, $\underline \theta$ through its respective norms. It is independent of $\varepsilon$.

\bigskip
\noindent{\it Step 3} Using \eqref{continuityeq-1D}  and \eqref{momentumeq-1D}  we get  

$$%\int_{\Omega_\ep} \rho (\de_t \tilde \u + \tilde  \u \cdot \nabla_x \tilde \u)\cdot (\tilde \u-\u) dx = 
\int_{\Omega_\ep} \rho (\de_t \tilde u + \tilde u \de_y \tilde u) (\tilde u-u^3) dx$$

$$ =\int_{\Omega_\ep} \frac{\rho}{\tilde \rho} (\tilde u -  u^3)  \left(\de_y[(\nu_0 + \nu_1 \tilde \theta) \de_y \tilde u] - \de_y p (\tilde \rho, \tilde \theta)\right)dx $$

$$ =\int_{\Omega_\ep} \frac{1}{\tilde \rho} (\rho-\tilde \rho) (\tilde u -  u^3)  \left(\de_y[(\nu_0 + \nu_1 \tilde \theta) \de_y \tilde u] - \de_y p (\tilde \rho, \tilde \theta)\right)dx $$

\begin{equation} +\int_{\Omega_\ep}  (\tilde u -  u^3)  \left(\de_y[(\nu_0 + \nu_1 \tilde \theta) \de_y \tilde u] - \de_y p (\tilde \rho, \tilde \theta)\right)dx. \label{step3ineq}\end{equation}
Due to the regularity of $(\tilde u,\tilde \rho, \tilde\theta)$ it follows that

\begin{multline*}
\int_{\Omega_\ep} \frac{1}{\tilde \rho} (\rho-\tilde \rho) (\tilde u -  u^3)  \left(\de_y[(\nu_0 + \nu_1 \tilde \theta) \de_y \tilde u] - \de_y p (\tilde \rho, \tilde \theta)\right)dx \\
\leq c \int_{\Omega_\ep}|(\rho - \tilde \rho)(\tilde u - u^3)|dx.
\end{multline*}
To estimate $\int_{\Omega_\ep}|(\rho - \tilde \rho)(\tilde u - u^3)|dx$ we introduce the following proposition. 

\begin{proposition}
There exists $c>0$ independent of $\rho$ and $\theta$ such that 

\begin{equation}\label{rustE}
\begin{array}{ll}
\mathcal E(\rho,\theta|\tilde\rho, \tilde \theta)  \geq c|\rho - \tilde \rho|^2 &\quad \mbox{ for } \rho\in[\underline \rho, \overline \rho],\\
\mathcal E(\rho,\theta|\tilde \rho, \tilde \theta) \geq c|\rho - \tilde \rho| &\quad \mbox{ for } \rho<\underline \rho,\\
\mathcal E(\rho,\theta|\tilde \rho, \tilde \theta) \geq c\rho &\quad \mbox{ for }\rho>\overline \rho.
\end{array}
\end{equation}
\end{proposition}
\begin{proof}
We use the notation from the proof of Lemma 5.1 in \cite{Fei2}, i.e., 

$$\mathcal F(\rho) = H^{\tilde \theta}(\rho, {\tilde \theta}) - \partial_\rho H^{\tilde \theta}(\tilde \rho, {\tilde \theta})(\rho - \tilde \rho) - H^{\tilde \theta}(\tilde \rho, {\tilde \theta}),$$
and
$$\mathcal G(\rho, \theta) = H^{\tilde \theta}(\rho,\theta) -H^{\tilde \theta}( \rho, {\tilde \theta}).$$

According to Section 2.2.3 in \cite{Fei2} it holds

$$\mathcal E(\rho,\theta|\tilde \rho, \tilde \theta) = \mathcal F(\rho) + \mathcal G(\rho,\theta)\geq \mathcal F(\rho) + \mathcal G(\rho,\tilde \theta)\geq \mathcal F(\rho),$$
and the function $\mathcal F(\rho)$ is strictly convex attaining its minimum $0$ when $\rho =\tilde \rho$. This immediately implies the first inequality in \eqref{rustE}.

As far as $\partial_\rho \mathcal F(r)\leq -c <0$ for $r\in \left(0,\frac32\underline \rho\right)$, we have

$$
\mathcal F(\rho)  = \mathcal F(\rho) - \mathcal F\left(\frac32\underline \rho\right) + \mathcal F\left(\frac32\underline \rho\right) - \mathcal F(\tilde \rho).
$$
Since $\mathcal F(\rho) - \mathcal F\left(\frac32\underline \rho\right)$ as well as $\mathcal F\left(\frac32\underline \rho\right) - \mathcal F(\tilde \rho)$ are positive, we may proceed as follows

$$\mathcal F(\rho) = \left|\mathcal F(\rho) - \mathcal F\left(\frac32\underline \rho\right)\right| + \left|\mathcal F\left(\frac32\underline \rho\right) - \mathcal F(\tilde \rho)\right| \geq c\left|\rho - \frac32\underline\rho\right|\geq |\rho - \tilde \rho|.$$

Similarly, {$\partial_\rho\mathcal F(r)\geq c>0$ for $r\in (\frac 23\overline \rho,\infty)$} and we deduce 

\begin{multline*}
\mathcal F(\rho) = \mathcal F(\rho)-\mathcal F(\tilde \rho) = \mathcal F(\rho)-\mathcal F\left(\frac 23 \overline \rho\right) + \mathcal F\left(\frac 23 \overline \rho\right) - \mathcal F(\tilde \rho)\\ \geq {c\left(\rho - \frac 23 \overline \rho\right) + c\left(\frac 23 \overline \rho - \tilde \rho\right)^2 \geq \frac{c}3\rho}.
\end{multline*}
\end{proof}

Now we estimate $\displaystyle \int_{\Omega_\ep}|(\rho - \tilde \rho)(\tilde u - u^3)|dx$. First, we split it into three parts:

\begin{multline*}\int_{\Omega_\varepsilon} |(\rho - \tilde \rho)(\tilde u - u^3)|dx = \int_{\{\underline \rho \leq \rho\leq \overline \rho\}} |\left(\rho - \tilde\rho\right) (\tilde u - u^3)|dx\\
+ \int_{\{\rho<\underline \rho\}}|\left(\rho - \tilde \rho\right)(\tilde u - u^3)|dx
+\int_{\{\rho>\overline \rho\}} |\left(\rho - \tilde\rho\right) (\tilde u - u^3)|dx. 
\end{multline*}
Second,  using \eqref{rustE} we estimate each term as follows: 

\begin{multline*}
\int_{\{\underline \rho\leq \rho\leq\overline \rho\}}|\left(\rho - \tilde \rho\right)(\tilde u - u^3)|dx\leq  \delta \|\tilde u - u^3\|_{L^2(\Omega_\varepsilon)}^2 + K(\delta, \cdot) \int_{\Omega_\varepsilon}\mathcal E(\rho,\theta|\tilde \rho, \tilde \theta)dx,
\end{multline*}

\begin{multline*}
\int_{\{\rho<\underline \rho\}}|\left(\rho - \tilde \rho\right)(\tilde u - u^3)|dx\leq \delta\int_{\Omega_\varepsilon} |\tilde u - u^3|^2dx + K(\delta,\cdot) \int_{\{\rho<\underline \rho\}}\left|\rho - \tilde \rho\right|^2dx \\ \leq \delta \|\tilde u - u^3\|_{L^2(\Omega_\varepsilon)}^2 + K(\delta,\cdot) \int_{\{\rho<\underline \rho\}}\left|\rho - \tilde \rho\right|dx \\ \leq \delta \|\tilde u - u^3\|_{L^2(\Omega_\varepsilon)}^2 + K(\delta,\cdot) \int_{\Omega_\varepsilon}\mathcal E(\rho,\theta|\tilde \rho, \tilde \theta)dx,
\end{multline*}
 for any $\delta >0$ and

\begin{multline*}
\int_{\{\rho>\overline \rho\}} \frac {|\rho - \tilde \rho|}{\sqrt \rho} \sqrt \rho |\tilde u - u^3|dx\leq \int_{\Omega_\varepsilon} \rho |\tilde u - u^3|^2dx + \int_{\{\rho>\overline \rho\}} \frac{|\rho - \tilde \rho|^2}{\rho}dx \\\leq \int_{\Omega_\varepsilon} \rho| \tilde u - u^3|^2 + K(\cdot)\int_{\{\rho>\overline \rho\}} \rho dx  \leq \int_{\Omega_\varepsilon} \rho|\tilde u - u^3|^2dx + K(\cdot)\int_{\Omega_\varepsilon}\mathcal E(\rho, \theta|\tilde \rho, \tilde \theta)dx.
\end{multline*}

The second integral on the right-hand side of \eqref{step3ineq}  is handled by integration by parts as 

\begin{multline*}\int_{\Omega_\ep}  (\tilde u -  u^3)  \left(\de_y[(\nu_0 + \nu_1 \tilde \theta) \de_y \tilde u] - \de_y p (\tilde \rho, \tilde \theta)\right)dx \\ = \int_{\Omega_\ep} (\nu_0 + \nu_1 \tilde \theta) \de_y \tilde u \de_y (u^3 - \tilde u)    +  p (\tilde \rho, \tilde \theta)   \de_y(\tilde u -  u^3)     dx .\end{multline*}

Thus we arrive at

$$\int_{\Omega_\ep} \rho (\de_t \tilde u + \tilde  u \partial_y \tilde u)(\tilde u- u^3) dx \leq \int_{\Omega_\ep} (\nu_0 + \nu_1 \tilde \theta) \de_y \tilde u \de_y(u^3 - \tilde u)    +  p (\tilde \rho, \tilde \theta)   \de_y(\tilde u -  u^3)     dx$$

\eqn{step3finaleq}
$$ + { \delta  }{\|u^3 - \tilde u \|_{L^2(\Omega_\ep)}^2} + 
K(\delta,\cdot)  \int_{\Omega_\ep} \left( \frac 12 \rho|\tilde u -u^3|^2+ \EE(\rho, \theta|\tilde \rho , \tilde \theta) \right)dx,$$
for any $\delta > 0$.

\bigskip
\noindent{\it Step 4} We calculate
$$ \int_{\Omega_\ep} \rho\left(s (\rho,  \theta) - s (\tilde \rho,\tilde \theta)\right)(\de_t \tilde \theta + \tilde u \partial_y\tilde \theta)dx $$

$$ =  \int_{\Omega_\ep} \rho\left[s (\rho,  \theta) - s (\tilde \rho,\tilde \theta)\right]_{ess}(\de_t \tilde \theta + \tilde u \de_y\tilde \theta) dx 
 + \int_{\Omega_\ep} \rho\left[s (\rho,  \theta) - s (\tilde \rho,\tilde \theta)\right]_{res}(\de_t \tilde \theta + \tilde u \de_y\tilde \theta) dx.$$
By a combination of \eqref{entfuncest} and \eqref{resestrho}  we get

$$\left| \int_{\Omega_\ep} \rho\left[s (\rho,  \theta) - s (\tilde \rho,\tilde \theta)\right]_{res}(\de_t \tilde \theta + \tilde u \de_y\tilde \theta) dx\right| \leq K(\cdot)\int_{\Omega_\ep} \EE (\rho, \theta|\tilde \rho, \tilde \theta)dx,$$
while

$$\int_{\Omega_\ep} \rho\left[s (\rho,  \theta) - s (\tilde \rho,\tilde \theta)\right]_{ess}(\de_t \tilde \theta + \tilde u \de_y\tilde \theta) dx = \int_{\Omega_\ep} (\rho-\tilde \rho)\left[s (\rho,  \theta) - s (\tilde \rho,\tilde \theta)\right]_{ess}(\de_t \tilde \theta + \tilde u \de_y\tilde \theta) dx $$

$$ + \int_{\Omega_\ep} \tilde \rho\left[s (\rho,  \theta) - s (\tilde \rho,\tilde \theta)\right]_{ess}(\de_t \tilde \theta + \tilde u \de_y\tilde \theta) dx ,$$
where with the help of the Taylor-Lagrange formula and \eqref{essest},

$$\left|\int_{\Omega_\ep} (\rho-\tilde \rho)\left[s (\rho,  \theta) - s (\tilde \rho,\tilde \theta)\right]_{ess}(\de_t \tilde \theta + \tilde u \de_y\tilde \theta) dx \right| $$

$$\leq \left(  \underset{(\rho, \theta) \in [\underline \rho, \overline \rho]\times [\underline \theta, \overline \theta]}{\sup} |\de _\rho s( \rho, \theta)| + \underset{(\rho, \theta) \in [\underline \rho, \overline \rho]\times [\underline \theta, \overline \theta]}{\sup} |\de _\theta s (\rho, \theta)| \right) \|\de_t \tilde \theta + \tilde u \de_y\tilde \theta\|_{L^\infty(0,1)}$$

\eqn{step4eq2} $$\times \int_{\Omega_\ep} |[\rho -\tilde \rho]_{ess}|(|[\rho - \tilde \rho]_{ess}| + |[\theta - \tilde \theta]_{ess}|)dx \leq K(\cdot)\int_{\Omega_\ep} \EE (\rho, \theta|\tilde \rho, \tilde \theta)dx .$$
Finally, we write 

$$\int_{\Omega_\ep} \tilde \rho\left[s (\rho,  \theta) - s (\tilde \rho,\tilde \theta)\right]_{ess}(\de_t \tilde \theta + \tilde u \de_y\tilde \theta) dx $$

$$ = 
\int_{\Omega_\ep}\tilde \rho\left[s (\rho,  \theta) - \de_\rho s (\tilde \rho, \tilde \theta)(\rho - \tilde \rho) - \de_\theta s (\tilde \rho, \tilde \theta) (\theta - \tilde \theta) - s (\tilde \rho,\tilde \theta)\right]_{ess}(\de_t \tilde \theta + \tilde u \de_y\tilde \theta) dx$$

$$ - \int_{\Omega_\ep} \tilde \rho\left[ \de_\rho s (\tilde \rho, \tilde \theta)(\rho - \tilde \rho) + \de_\theta s (\tilde \rho, \tilde \theta) (\theta - \tilde \theta) \right]_{res}(\de_t \tilde \theta + \tilde u \de_y\tilde \theta) dx$$

$$ + \int_{\Omega_\ep} \tilde \rho\left( \de_\rho s (\tilde \rho, \tilde \theta)(\rho - \tilde \rho) + \de_\theta s (\tilde \rho, \tilde \theta) (\theta - \tilde \theta) \right)(\de_t \tilde \theta + \tilde u \de_y\tilde \theta) dx,$$
where the first integral on the right-hand side can be estimated using the Taylor-Lagrange formula of the second order and \eqref{essest} (similarly to \eqref{step4eq2}) and the second integral on the right-hand side can be estimated by  \eqref{resestheta} and \eqref{resestrho}.  

Thus we conclude that

$$ -  \int_{\Omega_\ep}\rho\left(s (\rho,  \theta) - s (\tilde \rho,\tilde \theta)\right)(\de_t \tilde \theta + \tilde u \partial_y\tilde \theta) dx \leq K(\cdot)\int_{\Omega_\varepsilon} \EE(\rho, \theta|\tilde \rho, \tilde \theta) dx$$

\eqn{step4finaleq}
$$ - \int_{\Omega_\ep} \tilde \rho\left( \de_\rho s (\tilde \rho, \tilde \theta)(\rho - \tilde \rho) + \de_\theta s (\tilde \rho, \tilde \theta) (\theta - \tilde \theta) \right)(\de_t \tilde \theta + \tilde u  \de_y\tilde \theta) dx.$$

\bigskip
\noindent{\it Step 5} By the integration by parts
$$\int_{\Omega_\ep} \left(\left(1-\frac{\rho}{\tilde \rho}\right)\de_t p(\tilde \rho, \tilde \theta) - \frac\rho{\tilde \rho} \u \cdot \nabla_x p (\tilde \rho, \tilde \theta)\right)dx
$$

$$= \int_{\Omega_\ep} \left(\left(1-\frac{\rho}{\tilde \rho}\right)\de_t p(\tilde \rho, \tilde \theta) - \frac\rho{\tilde \rho} u^3 \de_y p (\tilde \rho, \tilde \theta)\right)dx$$

$$=  \int_{\Omega_\ep} (\tilde \rho - \rho) \frac1{\tilde \rho}\left(\de_t p(\tilde \rho, \tilde \theta) + \tilde  u \de_y p (\tilde \rho, \tilde \theta)\right)dx + \int_{\Omega_\ep} p (\tilde \rho, \tilde \theta) \de_y u^3 dx $$
$$+ \int_{\Omega_\ep} (\tilde \rho - \rho)\frac1{\tilde \rho} \de_y p (\tilde \rho, \tilde \theta) (u^3 - \tilde u) dx,$$
where, by means of the same arguments as in Step 3, 

\begin{multline*}\left| \int_{\Omega_\ep} (\tilde \rho - \rho)\frac1{\tilde \rho} \de_y p (\tilde \rho, \tilde \theta) (u^3 - \tilde u) dx\right| \\ \leq 
 { \delta }{\|u^3 - \tilde u \|_{L^2(\Omega_\ep)}^2} + K(\delta,\cdot) \int_{\Omega_\ep}\left( \frac12\rho |u^3 -\tilde u|^2  + \EE(\rho, \theta|\tilde \rho , \tilde \theta) \right)dx,\end{multline*}
for any $\delta > 0$. Using this estimate we get

\begin{multline}\label{step6finaleq}
\int_{\Omega_\ep} \left(\left(1-\frac{\rho}{\tilde \rho}\right)\de_t p(\tilde \rho, \tilde \theta) - \frac\rho{\tilde \rho} \u \cdot \nabla_x p (\tilde \rho, \tilde \theta)\right)dx \\ \leq \int_{\Omega_\ep} (\tilde \rho - \rho) \frac1{\tilde \rho}\left(\de_t p(\tilde \rho, \tilde \theta) + \tilde  u \de_y p (\tilde \rho, \tilde \theta)\right)dx 
+ \int_{\Omega_\ep} p (\tilde \rho, \tilde \theta) \de_y u^3 dx \\ + 
  { \delta }{\|u^3 - \tilde u \|_{L^2(\Omega_\ep)}^2} + K(\delta,\cdot) \int_{\Omega_\ep}\left( \frac12\rho |u^3 - \tilde u|^2+ \EE(\rho, \theta|\tilde \rho , \tilde \theta) \right)dx,\end{multline}
for any $\delta > 0$.

\bigskip
\noindent {\it Step 6} 
Summing up the estimates \eqref{step0finaleq}, \eqref{step1finaleq}, \eqref{step2finaleq}, \eqref{step3finaleq}, \eqref{step4finaleq} and \eqref{step6finaleq} we can rewrite the relative entropy inequality \eqref{relentineq2} in the form 

$$ \frac1{|Q_\ep|}\int_{\Omega_\ep} \left(\frac12 \rho | \u - \tilde \u|^2 + \EE(\rho, \theta|\tilde \rho,\tilde\theta)\right)(\tau , \cdot )dx $$

$$+\frac1{|Q_\ep|} \int_0^\tau  \int_{\Omega\ep} \left(\frac{\tilde \theta}{\theta} \S(\theta, \nabla_x \u):\nabla_x \u -  (\nu_0 + \nu_1\tilde \theta) \de_y \tilde u \de_y(u^3 - \tilde u) - \S (\theta, \nabla_x \u):\nabla_x \tilde \u \right)dx dt $$

$$
+\frac1{|Q_\ep|} \int_0^\tau  \int_{\Omega\ep} \left( \frac{\q (\theta, \nabla_x \theta)}{\theta} \cdot \nabla_x \tilde \theta  - \frac{\tilde \theta}{\theta} \frac{\q(\theta, \nabla_x \theta)\cdot \nabla_x \theta}{\theta}\right)dx dt  \leq  {\Gamma(\ep)}$$

$$+ \frac1{|Q_\ep|} \int_0^\tau  \left[  { \delta }{\|u^3 - \tilde u \|_{L^2(\Omega_\ep)}^2} +  K(\delta, \cdot)  \int_{\Omega_\ep}\left( \frac12\rho |u^3 - \tilde u|^2+ \EE(\rho, \theta|\tilde \rho , \tilde \theta) \right)dx\right]dt$$

$$
 + \frac1{|Q_\ep|} \int_0^\tau \int_{\Omega_\ep}\left( \delta \|\partial_y(u^3 - \tilde u)\|_{L^2(\Omega_\ep)}^2 + \ep \|\nabla \u\|_{L^2(\Omega_\ep)}^2\right)dx dt
$$

$$ +\frac1{|Q_\ep|}\int_0^\tau  \int_{\Omega_\ep} \left( p (\tilde \rho, \tilde \theta) - p (\rho, \theta)\right) \de_y\tilde u dx dt$$

$$+ \frac1{|Q_\ep|}\int_0^\tau   \int_{\Omega_\ep} (\tilde \rho - \rho) \frac1{\tilde \rho}\left(\de_t p(\tilde \rho, \tilde \theta) + \tilde  u \de_y p (\tilde \rho, \tilde \theta)\right)dx dt  $$

\eqn{relentineq3} 
$$-\frac1{|Q_\ep|}\int_0^\tau  \int_{\Omega_\ep} \tilde \rho\left( \de_\rho s (\tilde \rho, \tilde \theta)(\rho - \tilde \rho) + \de_\theta s (\tilde \rho, \tilde \theta) (\theta - \tilde \theta) \right)( \de_t \tilde \theta + \tilde u  \de_y \tilde \theta )   dxdt,$$
for almost all $\tau \in [0,T]$ and any $\delta >0$. Recall that {$\Gamma(\ep) \to 0$ as $\ep \to 0$}.

\bigskip
\noindent {\it Step 7}
Our next goal is to control the last three integrals on the right-hand side of \eqref{relentineq3}.  To this end, we recall a useful identity that follows directly from the Gibbs' equation \eqref{gibbs}:

{\eqn{gibbsident2} $$ r \de _\rho s(r,\Theta) = - \frac1r \de_\theta p(r,\Theta).$$}

Using \eqref{gibbsident2} we obtain
 $$\int_{\Omega_\ep} (\tilde \rho - \rho) \frac1{\tilde \rho}\left(\de_t p(\tilde \rho, \tilde \theta) + \tilde  u \de_y p (\tilde \rho, \tilde \theta)\right)dx
$$

$$-\int_{\Omega_\ep} \tilde \rho\left( \de_\rho s (\tilde \rho, \tilde \theta)(\rho - \tilde \rho) + \de_\theta s (\tilde \rho, \tilde \theta) (\theta - \tilde \theta) \right)( \de_t \tilde \theta + \tilde u  \de_y \tilde \theta )   dx$$

 $$= \int_{\Omega_\ep} \tilde \rho  (\tilde \theta - \theta)  \de_\theta s (\tilde \rho, \tilde \theta) ( \de_t \tilde \theta + \tilde u  \de_y \tilde \theta )   dx
+ \int_{\Omega_\ep}  (\tilde \rho - \rho)  \frac1{\tilde \rho}  \de_\rho p (\tilde \rho, \tilde \theta) ( \de_t \tilde \rho + \tilde u  \de_y \tilde \rho )   dx.$$
Since $\tilde \rho$ and  $\tilde u$ satisfy the equation of continuity \eqref{continuityeq-1D} we get 

 $$\int_{\Omega_\ep}  (\tilde \rho - \rho)  \frac1{\tilde \rho}  \de_\rho p (\tilde \rho, \tilde \theta) ( \de_t \tilde \rho + \tilde u  \de_y \tilde \rho )   dx = - \int_{\Omega_\ep}  (\tilde \rho - \rho)    \de_\rho p (\tilde \rho, \tilde \theta) \de_y \tilde u     dx.$$
By using \eqref{gibbsident2} and \eqref{continuityeq-1D} once more,  followed by the use of the entropy equation \eqref{entropyeq-1D}  we deduce that 

$$\int_{\Omega_\ep} \tilde \rho  (\tilde \theta - \theta)  \de_\theta s (\tilde \rho, \tilde \theta) ( \de_t \tilde \theta + \tilde u  \de_y \tilde \theta )   dx$$

$$ = \int_{\Omega_\ep}   \tilde \rho(\tilde \theta - \theta) \left[ \de_t s(\tilde \rho, \tilde \theta) + \tilde u \de_y s(\tilde \rho, \tilde \theta)\right]  dx - \int_{\Omega_\ep}   (\tilde \theta - \theta)  \de_\theta p (\tilde \rho, \tilde \theta) \de_y \tilde u   dx$$

$$ = \int_{\Omega_\ep}   (\tilde \theta - \theta)\left[   \frac1{\tilde \theta} \left((\nu_0 + \nu_1 \tilde \theta) (\de_y \tilde u)^2 - \frac{q(\tilde \theta, \de_y \tilde \theta)\de_y \tilde \theta}{{\tilde \theta}} \right)  -\de_y\left( \frac{q(\tilde \theta,\de_y \tilde \theta)}{\tilde \theta}\right)  \right]  dx $$

 $$- \int_{\Omega_\ep}   (\tilde \theta - \theta)  \de_\theta p (\tilde \rho, \tilde \theta) \de_y \tilde u   dx.$$

We  use the Taylor-Lagrange formula together with \eqref{essest} on the essential part;  \eqref{entfuncest} with $p(\rho, \theta) \leq c\rho e(\rho, \theta)$ and  \eqref{resestheta}, \eqref{resestrho} on the residual part to obtain 
$$\left| \int_{\Omega_\ep} \left( p (\tilde \rho, \tilde \theta) -      \de_\rho p (\tilde \rho, \tilde \theta) (\tilde \rho - \rho)    -    \de_\theta p (\tilde \rho, \tilde \theta) (\tilde \theta - \theta) - p (\rho, \theta) \right) \de_y \tilde u   dx\right|$$

$$\leq K(\cdot) \int_{\Omega_\varepsilon} \EE(\rho, \theta|\tilde \rho, \tilde \theta) dx.$$

Finally, the integration by parts with \eqref{bc1D} allows us to rewrite \eqref{relentineq3}  in the form

$$ \frac1{|Q_\ep|}\int_{\Omega_\ep} \left(\frac12 \rho | \u - \tilde \u|^2 + \EE(\rho, \theta|\tilde \rho,\tilde\theta)\right)(\tau , \cdot )dx $$

$$+\frac1{|Q_\ep|} \int_0^\tau  \int_{\Omega\ep} \left(\frac{\tilde \theta}{\theta} \S(\theta, \nabla_x \u):\nabla_x \u -  (\nu_0 +\nu_1\tilde \theta) \de_y \tilde u \de_y(u^3 - \tilde u) - \S (\theta, \nabla_x \u):\nabla_x \tilde \u \right.$$

$$\left. - \frac{\tilde \theta - \theta}{\tilde \theta} (\nu_0 +\nu_1\tilde \theta)(\de_y \tilde u)^2\right)dx dt $$

$$
+\frac1{|Q_\ep|} \int_0^\tau  \int_{\Omega\ep} \left( \frac{\q (\theta, \nabla_x \theta)\cdot \nabla_x \tilde \theta}{\theta}   - \frac{\tilde \theta}{\theta} \frac{\q(\theta, \nabla_x \theta)\cdot \nabla_x \theta}{\theta}\right.$$

$$\left.+ (\tilde \theta - \theta)    \frac{q(\tilde \theta, \de_y \tilde \theta)\de_y \tilde \theta}{{\tilde \theta}^2}+ \de_y  (\theta - \tilde \theta )    \frac{q(\tilde \theta, \de_y \tilde \theta)}{\tilde \theta}  \right)dx dt $$

  $$\leq  {\Gamma(\ep)} + \frac1{|Q_\ep|} \int_0^\tau  \left[  { \delta }{\|u^3 - \tilde u \|_{L^2(\Omega_\ep)}^2}   +  K(\delta, \cdot) \int_{\Omega_\ep}\left( \frac12\rho |u^3 - \tilde u|^2+ \EE(\rho, \theta|\tilde \rho , \tilde \theta) \right)dx\right] dt$$

\eqn{relentineq4}	$$
 + \frac1{|Q_\ep|} \int_0^\tau \int_{\Omega_\ep}\left( \delta \|\partial_y(u^3 - \tilde u)\|_{L^2(\Omega_\ep)}^2 + \ep \|\nabla \u\|_{L^2(\Omega_\ep)}^2 \right)dx dt,
$$
for almost all $\tau \in [0,T]$ and any $\delta >0$.

\section{Dissipative terms}

The goal of this section is to show that the "dissipative" terms on the left-hand side of \eqref{relentineq4}   containing $\nabla_x \u$ are strong enough
to control   the velocity terms on the right-hand side.

\subsection{Viscosity}

In accordance with hypotheses \eqref{const-stress} and  \eqref{stress-cond} we have 

$$\mathbb{S} (\theta, \nabla_x \ve) = \mathbb{S}^0( \nabla_x \ve) +  \theta \mathbb{S}^1 (\nabla_x \ve),$$ with  

$$\mathbb{S}^i (\nabla_x \ve) = \mu_i (\nabla_x \ve + \nabla_x^T  \ve - \frac23 \div_x \ve \mathbb{I}) + \eta_i \div_x \ve \mathbb{I}, \ \ \ i =0,1.$$

 Using Lemma \ref{kornlike} we immediately obtain
\begin{equation}\label{c:intodhad}
\int_{\Omega_\ep} \S^i(\nabla_x\ve) : \nabla_x\ve \dx  \geq \int_{\Omega_\ep} \mu_i \abs{\nabla_x \ve}^2 + \eta_i\abs{\div_x\ve}^2 \dx,
\end{equation}
for $i=0,1$.

\subsection{Viscosity terms}

The terms to be dealt with are

\begin{multline*}\int_{\Omega\ep} \left(\frac{\tilde \theta}{\theta} \S(\theta, \nabla_x \u):\nabla_x \u -  (\nu_0 + \nu_1\tilde \theta) \de_y \tilde u \de_y(u^3 - \tilde u)\right.\\ - \left.\S (\theta, \nabla_x \u):\nabla_x \tilde \u  - \frac{\tilde \theta - \theta}{\tilde \theta} (\nu_0 + \nu_1\tilde \theta) (\de_y \tilde u)^2\right)dx.
\end{multline*}

First we deal with the "$\theta \S^1(\nabla_x \u)$" part: 

$$ \frac{\tilde \theta}{\theta} \theta \S^1(\nabla_x \u) :\nabla_x \u - \tilde \theta \S^1( \nabla_x \tilde \u) : (\nabla_x \u - \nabla_x \tilde \u) - \theta \S^1( \nabla_x \u):\nabla_x \tilde \u - \frac{ \tilde \theta - \theta}{\tilde \theta} \tilde \theta \S^1( \nabla_x \tilde \u ) : \nabla_x \tilde \u $$

$$= \tilde \theta \left( \S^1(\nabla_x \u)  - \S^1( \nabla_x \tilde \u)\right) : (\nabla_x \u - \nabla_x \tilde \u) +(\tilde \theta - \theta) \left( \S^1( \nabla_x \u) -  \S^1( \nabla_x \tilde \u )\right) : \nabla_x \tilde \u. $$
Using \eqref{c:intodhad} we obtain

$$\int_{\Omega_\varepsilon} \tilde \theta \left(\S^1(\nabla_x \u) - \S^1( \nabla_x \tilde \u) \right): (\nabla_x  \u -\nabla_x \tilde \u) dx \geq \mu_1 \underline{\theta} \int_{\Omega_\varepsilon} |\nabla_x(\u  -\tilde \u) |^2 dx . $$ 
Since 

$$ u^3(\cdot, y ) = \tilde u (\cdot, y) = 0 \mbox{ for } y=0,1,$$
we see  that

\begin{multline*}
\int_{\Omega_\varepsilon} \tilde \theta \left(\S^1(\nabla_x \u) - \S^1(\nabla_x \tilde \u) \right): (\nabla_x  \u -\nabla_x \tilde \u) dx \\ \geq c \left( \| u^3 -  \tilde u\|_{L^{2}(\Omega_\ep)}^2  + \|\nabla_x (\u - \tilde \u)\|^2_{L^2(\Omega_\ep)}\right).\end{multline*}

 Similarly using \eqref{esthetanon} we can show that

\begin{multline}
\left|\int_{\Omega_\varepsilon} (\tilde \theta - \theta) \left( \S^1( \nabla_x \u) - \S^1( \nabla_x \tilde \u)\right):\nabla_x \tilde \u dx\right| \\
\leq \delta \left(\|\de_y( u^3 - \tilde u)\|_{L^2(\Omega_\ep)}^2 {+\|\div_x (\u - \tilde \u)\|_{L^2(\Omega_\ep)}^2}\right) + K(\delta, \cdot)  \| \theta - \tilde \theta \|_{L^2(\Omega_\varepsilon)}^2,\end{multline}
for any $\delta> 0$.

Since 
\begin{equation}
\label{viscosityest7}
 \S^i( \nabla_x \tilde \u ) : \nabla_x \tilde \u = \nu_i (\de_y \tilde u)^2, \ i=0,1,
\end{equation}
we have that 

\begin{equation*}\frac{ \tilde \theta - \theta}{\tilde \theta} \tilde \theta \S^1( \nabla_x \tilde \u ) : \nabla_x \tilde \u = \frac{\tilde \theta - \theta}{\tilde \theta} \nu_1 \tilde \theta  (\de_y \tilde u)^2.\end{equation*}

Finally, we estimate the difference

$$\int_{\Omega_\ep} \left(\tilde \theta \S^1( \nabla_x \tilde \u) : (\nabla_x \u - \nabla_x \tilde \u) -   \nu_1 \tilde \theta  \de_y \tilde u \de_y(u^3 - \tilde u)\right)dx.$$
Proceeding with the integration by parts

$$\int_{\Omega\ep}  \tilde \theta\S^1( \nabla_x \tilde \u) : (\nabla_x \u - \nabla_x \tilde \u) dx = - \int_{\Omega\ep}\div_x [ \tilde \theta \S^1( \nabla_x \tilde \u)] \cdot ( \u - \tilde \u) dx$$

$$ + \int_{\partial \Omega_\ep}     \tilde \theta \S^1( \nabla_x \tilde \u)( \u - \tilde \u) \n dS.$$
Since 

\begin{multline*}- \int_{\Omega_\ep} \div_x [\tilde \theta  \S^1( \nabla_x \tilde \u)] \cdot ( \u - \tilde \u) dx = - \int_{\Omega_\ep} \de_y[ \tilde \theta \nu_1 \de_y \tilde u] (u^3 - \tilde u)dx \\=\int_{\Omega_\ep}  \tilde \theta \nu_1 \de_y \tilde u \de_y(u^3 - \tilde u)dx  ,\end{multline*}
is is only left to estimate the boundary integral. 
By the same arguments that lead to \eqref{viscosityboundary0} later, we see that

$$\int_{\partial \Omega_\ep}     \tilde \theta\S^1( \nabla_x \tilde \u)( \u - \tilde \u) \n dS= 0.$$

Thus we get

$$ \int_{\Omega_\varepsilon} \tilde \theta \S^1( \nabla_x \tilde \u) : (\nabla_x \u - \nabla_x \tilde \u)dx= \int_{\Omega_\varepsilon}    \nu_1\tilde \theta \de_y \tilde u \de_y(u^3 - \tilde u) dx .  $$

Now we treat the "$\S^0(\nabla_x \u)$" part: We first suppose that $\theta \geq \tilde \theta$ and calculate

$$ \frac{\tilde \theta}{\theta} \S^0( \nabla_x \u) :\nabla_x \u - \S^0( \nabla_x \tilde \u) : (\nabla_x \u - \nabla_x \tilde \u) - \S^0( \nabla_x \u):\nabla_x \tilde \u $$

$$- \frac{ \tilde \theta - \theta}{\tilde \theta} \S^0( \nabla_x \tilde \u ) : \nabla_x \tilde \u \geq \frac{\tilde \theta}\theta (\S^0(\nabla_x \u) - \S^0(\nabla_x \tilde \u)):\nabla_x ( \u - \tilde \u) $$

$$+ \tilde \theta\left( \frac1\theta - \frac1{\tilde \theta}\right) \S^0 (\nabla_x \tilde \u):\nabla_x ( \u - \tilde \u) + \frac{ \tilde \theta - \theta}\theta ( \S^0 ( \nabla_x \u) - \S^0(\nabla_x \tilde \u)) :\nabla_x \tilde \u.$$
Since the function $\theta \mapsto \frac1\theta$ is Lipschitz on the set $\theta \geq \tilde \theta$, we conclude that 

$$\int_{\{\theta \geq \tilde \theta\}} \left| \tilde \theta \left(\frac1\theta - \frac1{\tilde \theta}\right) \S^0 (\nabla_x \tilde \u) :\nabla_x (\u - \tilde \u)\right| dx $$

$$ \leq K(\cdot)\|  \partial_y \tilde u \|_{L^\infty(0,1)}   \frac{\|\tilde \theta\|_{L^\infty(0,1)}}{\underline \theta^2}  \int_{\Omega_\ep} |\theta - \tilde \theta| (|\partial_y ( u^3 - \tilde u)| + |\div_x (\u - \tilde \u)|) dx$$

\eqn{viscosityest9}
$$\leq \delta   \left(  \|\partial_y (u^3 - \tilde u)\|_{L^2(\Omega_\ep)}^2  + \|\div_x (\u - \tilde \u)\|_{L^2(\Omega_\ep)}^2\right)  + K(\delta, \cdot) \|\theta - \tilde \theta\|_{L^2(\Omega_\varepsilon)}^2,$$
for any $\delta > 0$. 
Analogously, we obtain a similar estimate  for the last term,

$$\int_{\{\theta \geq \tilde \theta\}}  \left|\frac{ \tilde \theta - \theta}\theta ( \S^0 ( \nabla_x \u) - \S^0(\nabla_x \tilde \u)) :\nabla_x \tilde \u\right|dx$$

$$\leq \delta  \left( \|\partial_y (u^3 - \tilde u)\|_{L^2(\Omega_\ep)}^2 {+ \|\div_x (\u - \tilde \u)\|_{L^2(\Omega_\ep)}^2}\right)
 + K(\delta, \cdot)  {\|\theta - \tilde \theta\|_{L^2(\Omega_\varepsilon)}^2},$$
for any $\delta > 0$. 

Next, if {$0 < \theta \leq \tilde \theta$}, we have

\begin{multline*} \frac{\tilde \theta}{\theta} \S^0( \nabla_x \u) :\nabla_x \u -  \S^0( \nabla_x \tilde \u) : (\nabla_x \u - \nabla_x \tilde \u) - \S^0( \nabla_x \u):\nabla_x \tilde \u 
- \frac{ \tilde \theta - \theta}{\tilde \theta} \S^0( \nabla_x \tilde \u ) : \nabla_x \tilde \u\\
\geq (\S^0(\nabla_x \u ) - \S^0 (\nabla_x \tilde \u)):\nabla_x (\u - \tilde \u)
 + \frac{\tilde \theta- \theta}{\tilde \theta}[\S^0 (\nabla_x \u):\nabla_x \u - \S^0(\nabla_x \tilde \u):\nabla_x \tilde \u],
\end{multline*}
using $(\theta - \tilde \theta)^2 \geq 0$ and whence, by means of {convexity} of the function $\nabla_x \u \mapsto \S^0(\nabla_x \u):\nabla_x \u$ we get

$$\frac{\tilde \theta- \theta}{\tilde \theta}[\S^0 (\nabla_x \u):\nabla_x \u - \S^0(\nabla_x \tilde \u):\nabla_x \tilde \u] {\geq} 2 \frac{\tilde \theta- \theta}{\tilde \theta}\S^0 (\nabla_x \tilde \u):\nabla_x (\u - \tilde \u) ,$$
where similarly as in \eqref{viscosityest9}, we can estimate 

 $$\left |\int_{\Omega_\varepsilon} \frac{\tilde \theta- \theta}{\tilde \theta}\S^0 (\nabla_x \tilde \u):\nabla_x (\u - \tilde \u)dx\right|$$

$$\leq \delta   \left(  \|\partial_y (u^3 - \tilde u)\|_{L^2(\Omega_\ep)}^2  + \|\div_x (\u - \tilde \u)\|_{L^2(\Omega_\ep)}^2\right)  + K(\delta, \cdot) {\|\theta - \tilde \theta\|_{L^2(\Omega_\varepsilon)}^2},$$
for any $\delta > 0$.

Finally, we estimate the ``extra terms''. Thanks to \eqref{viscosityest7} we immediately get that

$$\frac{ \tilde \theta - \theta}{\tilde \theta} \S^0( \nabla_x \tilde \u ) : \nabla_x \tilde \u = \frac{\tilde \theta - \theta}{\tilde \theta} \nu_0 (\de_y \tilde u)^2.$$
Using the integration by parts  we calculate 

\begin{multline*}\int_{\Omega\ep}  \S^0( \nabla_x \tilde \u) : (\nabla_x \u - \nabla_x \tilde \u) dx = \\- \int_{\Omega\ep}  \div_x \S^0( \nabla_x \tilde \u) \cdot ( \u - \tilde \u) dx + \int_{\partial \Omega_\ep}     \S^0( \nabla_x \tilde \u) ( \u - \tilde \u) \n               dS.\end{multline*}
Since 

$$ ( \u - \tilde \u)\cdot \S^i( \nabla_x \tilde \u) \n  =  \de_y \tilde u \left(\begin{array}{c} (\eta_i - \frac23 \mu_i) n^1 \\ (\eta_i - \frac23 \mu_i) n^2 \\ \nu_i n^3 \end{array}\right) \cdot \left(\begin{array}{c} u^1 \\ u^2 \\ u^3 - \tilde u \end{array}\right), \mbox{ with } \n = \left(\begin{array}{c} n^1 \\  n^2 \\  n^3 \end{array}\right),$$  for $i=0,1$ 
 and the fact that  $$\n = \left(\begin{array}{c} 0 \\  0 \\  n^3 \end{array}\right), $$ on $ Q_\ep\times\{0\}$ and $ Q_\ep\times \{1\}$  and 

$$\n = \left(\begin{array}{c} n^1 \\  n^2 \\  0 \end{array}\right), $$
on $\de Q_\ep \times (0,1)$,  we get from the boundary conditions $\u\cdot \n|_{\partial \Omega_\varepsilon} = 0$ and  $\tilde u(\cdot , 0 ) = \tilde u(\cdot, 1) = 0$ that

\eqn{viscosityboundary0}
$$ \int_{\partial \Omega_\ep}   \S^0( \nabla_x \tilde \u)( \u - \tilde \u)  \n dS= 0.$$
Finally, it is straightforward to see that 

 $$ - \int_{\Omega\ep}  \div_x \S^0( \nabla_x \tilde \u) \cdot ( \u - \tilde \u) dx  = - \int_{\Omega\ep}   \nu_0 \de_y^2 \tilde u (u^3 - \tilde u)   dx=  \int_{\Omega\ep}   \nu_0 \de_y \tilde u \de_y(u^3 - \tilde u)   dx.$$
Hence,

$$\int_{\Omega\ep}  \S^0( \nabla_x \tilde \u) : (\nabla_x \u - \nabla_x \tilde \u) dx =
 \int_{\Omega\ep}   \nu_0 \de_y \tilde u \de_y(u^3 - \tilde u)   dx.$$

Summing up the results of this section together with \eqref{esthetanon} we can rewrite the relation \eqref{relentineq4} as

$$ \frac1{|Q_\ep|}\int_{\Omega_\ep} \left(\frac12 \rho | \u - \tilde \u|^2 + \EE(\rho, \theta|\tilde \rho,\tilde\theta)\right)(\tau , \cdot )dx $$

$$+c \frac1{|Q_\ep|} \int_0^\tau   \left( \| u^3 -  \tilde u\|_{L^{2}(\Omega_\ep)}^2  + \|\nabla_x (\u - \tilde \u)\|^2_{L^2(\Omega_\ep)}\right) dt $$

$$
+\frac1{|Q_\ep|} \int_0^\tau  \int_{\Omega\ep} \left( \frac{\q (\theta, \nabla_x \theta)\cdot \nabla_x \tilde \theta}{\theta}   - \frac{\tilde \theta}{\theta} \frac{\q(\theta, \nabla_x \theta)\cdot \nabla_x \theta}{\theta}\right.$$

$$\left.+ (\tilde \theta - \theta)    \frac{q(\tilde \theta, \de_y \tilde \theta)\de_y \tilde \theta}{{\tilde \theta}^2}+ \de_y  (\theta - \tilde \theta )    \frac{q(\tilde \theta, \de_y \tilde \theta)}{\tilde \theta}  \right)dx dt $$

 $$\leq  {\Gamma(\ep)} + \frac1{|Q_\ep|} \int_0^\tau  \left[ { \delta }\left( \|u^3 - \tilde u\|_{L^{2}(\Omega_\ep)}^2 +  \|\nabla_x (\u - \tilde \u)\|^2_{L^2(\Omega_\ep)}\right)   \right.  $$

\eqn{relentineq5} 
$$ \left. + \ep \|\nabla_x \u\|_{L^2(\Omega_\ep)}^2+K(\delta, \cdot) \int_{\Omega_\ep}\left( \frac12 \rho |u^3 - \tilde u |^2  +  \EE(\rho, \theta|\tilde \rho , \tilde \theta)\right) dx   \right] dt, $$
for almost all $\tau \in [0,T]$ and any $\delta >0$. At this moment we also point out that due to \eqref{totaldiss} and \eqref{ititiallimit}, \eqref{initlim2} we have that 

$$\frac 1{|Q_\ep|}\int_0^T \left\|\nabla_x \u + \nabla_x^T \u - \frac 23 \div_x \u \mathbb{I}\right\|^2_{L^2(\Omega_\ep)}dt\leq c,
$$
with $c$ independent of $\ep$ and thus by Lemma \ref{kornlike} the term

$$
\ep\frac 1{|Q_\ep|}\int_0^\tau  \|\nabla_x \u\|_{L^2(\Omega_\ep)}dt,
$$
can be included in $\Gamma(\varepsilon)$ (cf.  Equation (2.54) in \cite{Fei2}).

Now we may choose $\delta >0$ so small that the inequality  \eqref{relentineq5} takes the form

$$ \frac1{|Q_\ep|}\int_{\Omega_\ep} \left(\frac12 \rho | \u - \tilde \u|^2 + \EE(\rho, \theta|\tilde \rho,\tilde\theta)\right)(\tau , \cdot )dx $$

$$
+\frac1{|Q_\ep|} \int_0^\tau  \int_{\Omega\ep} \left( \frac{\q (\theta, \nabla_x \theta)\cdot \nabla_x \tilde \theta}{\theta}   - \frac{\tilde \theta}{\theta} \frac{\q(\theta, \nabla_x \theta)\cdot \nabla_x \theta}{\theta}\right.$$

$$\left.+ (\tilde \theta - \theta)    \frac{q(\tilde \theta, \de_y \tilde \theta)\de_y \tilde \theta}{{\tilde \theta}^2}+ \de_y  (\theta - \tilde \theta )    \frac{q(\tilde \theta, \de_y \tilde \theta)}{\tilde \theta}  \right)dx dt $$

\eqn{relentineq6} 
$$\leq \Gamma(\ep) +C\frac1{|Q_\ep|} \int_0^\tau  \int_{\Omega_\ep} \left(  \frac12 \rho |u^3 - \tilde u |^2 + \EE(\rho, \theta|\tilde \rho , \tilde \theta)   \right)dx  dt, $$
for almost all $\tau \in [0,T]$.

\subsection{Heat conductivity} In accordance with hypotheses \eqref{const-heatflux} and \eqref{heatflux-cond}

$$\q(\theta, \nabla_x \theta ) = -\kappa_0\nabla_x \theta  - \kappa_2 \theta^2\nabla_x \theta - \kappa_3 \theta^3 \nabla_x \theta,$$
and thus

$$q(\tilde \theta, \de_y \tilde \theta) = - \kappa_0 \partial_y \tilde \theta  - \kappa_2 \tilde \theta^2 \partial_y \tilde\theta - \kappa_3 \tilde \theta^3 \de_y \tilde \theta.$$

The terms to be dealt with are

$$
\frac{\q (\theta, \nabla_x \theta)\cdot \nabla_x \tilde \theta}{\theta}   - \frac{\tilde \theta}{\theta} \frac{\q(\theta, \nabla_x \theta)\cdot \nabla_x \theta}{\theta}+ (\tilde \theta - \theta)    \frac{q(\tilde \theta, \de_y \tilde \theta)\de_y \tilde \theta}{{\tilde \theta}^2}+ \de_y  (\theta - \tilde \theta )    \frac{q(\tilde \theta, \de_y \tilde \theta)}{\tilde \theta}. $$
We compute

$$\frac{\tilde \theta}{\theta} \frac{\kappa_0}\theta |\nabla_x \theta|^2 - \frac{\kappa_0}{\theta} \de_y \theta \de_y \tilde \theta + \frac{\theta - \tilde \theta}{\tilde \theta} \frac{\kappa_0}{\tilde \theta } |\de_y \tilde \theta|^2 + \frac{\kappa_0}{\tilde \theta} \de_y \tilde \theta \de_y (\tilde \theta - \theta)$$

$$ = \kappa_0 \left[\tilde \theta |\nabla_x \log (\theta)|^2 - \tilde \theta \de_y \log (\theta) \de_y \log (\tilde \theta) + (\theta - \tilde \theta)|\de_y \log (\tilde \theta)|^2 + \de_y \log (\tilde \theta) \de_y (\tilde \theta- \theta) \right]$$

$$ = \kappa_0 \left[\tilde \theta |\nabla_x \log (\theta)  - \nabla_x \log (\tilde \theta)         |^2  + (\theta - \tilde \theta)|\de_y \log (\tilde \theta)|^2 +\de_y\log (\tilde \theta) \cdot \de_y (\tilde \theta- \theta) \right.$$

$$\left.+ \tilde \theta \de_y \log (\tilde \theta)  (\de_y \log (\theta) - \de_y \log (\tilde \theta))\right] $$

$$= \kappa_0 \left[\tilde \theta |\nabla_x \log (\theta)  - \nabla_x \log (\tilde \theta)         |^2  + (\theta - \tilde \theta)|\de_y \log (\tilde \theta)|^2 + (\tilde \theta -  \theta)\de_y \log (\tilde \theta)\de_y \log (\theta) \right]$$

\eqn{heatcondeq1} $$= \kappa_0 \left[\tilde \theta |\nabla_x \log (\theta)  - \nabla_x \log (\tilde \theta)         |^2  + (\theta - \tilde \theta)\de_y \log (\tilde \theta)\de_y \left( \log (\tilde \theta) - \log (\theta) \right)\right]. $$

Similarly, we get 

$$\kappa_2 \tilde \theta |\nabla_x \theta|^2 - \kappa_2\theta \de_y \theta \de_y \tilde \theta + \kappa_2 (\theta - \tilde \theta) |\de_y \tilde \theta|^2 + \kappa_2\tilde \theta \de_y \tilde \theta \de_y (\tilde \theta - \theta)$$

\begin{equation}= \kappa_2 \left[ \tilde \theta |\nabla_x \theta - \nabla_x \tilde \theta|^2 - (\theta - \tilde \theta) \de_y \tilde \theta \de_y (\theta - \tilde \theta)\right].\end{equation}
Finally, 

$$\kappa_3 \theta \tilde \theta |\nabla_x \theta|^2 - \kappa_3\theta^2 \de_y \theta \de_y \tilde \theta + \kappa_3 (\theta - \tilde \theta) \tilde \theta  |\de_y \tilde \theta|^2 +\kappa_3\tilde \theta^2 \de_y \tilde \theta \de_y (\tilde \theta - \theta)$$

$$=\kappa_3 \left[ \tilde  \theta  \theta \nabla_x \theta\cdot (\nabla_x \theta - \nabla_x \tilde \theta) + \tilde \theta \theta \de_y \theta \de_y \tilde \theta -\theta^2 \de_y \theta \de_y \tilde \theta + (\theta - \tilde \theta) \tilde \theta  |\de_y \tilde \theta|^2 + \tilde \theta^2 \de_y \tilde \theta \de_y (\tilde \theta - \theta)\right]$$

$$=\kappa_3 \left[ \tilde  \theta  \theta |\nabla_x \theta - \nabla_x \tilde \theta|^2 + 2\tilde \theta \theta \de_y \theta \de_y \tilde \theta -\theta^2 \de_y \theta \de_y \tilde \theta  - \tilde \theta^2 \de_y \tilde \theta \de_y \theta\right]$$

\eqn{heatcondeq3}
$$=\kappa_3  \tilde  \theta  \theta |\nabla_x \theta - \nabla_x \tilde \theta|^2 - \kappa_3  (\theta- \tilde \theta)^2  \de_y \theta \de_y \tilde \theta ,$$
where

$$\kappa_3(\theta- \tilde \theta)^2  \de_y \theta \de_y \tilde \theta = \kappa_3|\de_y \tilde \theta|^2 (\theta - \tilde \theta)^2 + \kappa_3\de_y (\theta - \tilde \theta) \de_y \tilde \theta (\theta - \tilde \theta)^2.$$
We conclude by observing that the terms on the right-hand side of \eqref{heatcondeq1}--\eqref{heatcondeq3} have either "good" sign or they can be "absorbed" by the remaining integrals in \eqref{relentineq6}. Therefore,  using \eqref{heatcondeq1}--\eqref{heatcondeq3} we can rewrite the inequality \eqref{relentineq6} in the form\footnote{We emphasize that $(\rho,\u,\theta) = (\rho_\ep,\u_\ep,\theta_\ep)$ due to the convention from Section \ref{precal}.}

$$ \frac1{|Q_\ep|}\int_{\Omega_\ep} \left(\frac12 \rho | \u_\ep - \tilde \u|^2 + \EE(\rho_\ep, \theta_\ep|\tilde \rho,\tilde\theta)\right)(\tau , \cdot )dx  $$

$$\leq {\Gamma(\ep)}  + c\frac1{|Q_\ep|} \int_0^\tau \int_{\Omega_\ep} \left(\frac12 \rho |u_\ep^3 - \tilde u |^2 + \EE(\rho_\ep, \theta_\ep|\tilde \rho , \tilde \theta)  \right)dx         dt,$$
for almost all $\tau \in [0,T]$ with   {$\Gamma(\ep) \to 0$ as $\ep \to 0$}.

The Gronwall inequality yields

 $$ \frac1{|Q_\ep|}\int_{\Omega_\ep} \left(\frac12 \rho | \u_\ep - \tilde \u|^2 + \EE(\rho_\ep, \theta_\ep|\tilde \rho,\tilde\theta)\right)(\tau , \cdot )dx \leq \Gamma(\ep,T) ,$$
where $\Gamma(\ep,T) \to 0$ as $\ep \to 0$. Consequently, the main theorem is proven.

\end{document}